\documentclass{amsart}
\usepackage{verbatim}
\numberwithin{equation}{section}
\usepackage[all]{xy}
\usepackage{color}
\usepackage{amssymb}
\usepackage{stix}
\usepackage[left=3.5cm,right=3.5cm]{geometry}
\usepackage{tikz-cd}
\usepackage{mathtools}

\let\cal\mathcal
\def\Ascr{{\cal A}}
\def\Bscr{{\cal B}}
\def\Cscr{{\cal C}}
\def\Dscr{{\cal D}}

\def\Mscr{{\cal M}}

\def\Tscr{{\cal T}}

\let\blb\mathbb

\def \ZZ{{\blb Z}}

\def \HH{{\blb H}}

\def\id{\text{id}}
\def\Id{\operatorname{id}}

\def\Der{\operatorname{Der}}

\def\Ab{\mathbb{Ab}}

\def\mod{\operatorname{mod}}

\def\Ext{\operatorname {Ext}}
\def\Hom{\operatorname {Hom}}

\def\im{\operatorname {im}}

\def\ker{\operatorname {ker}}

\def\id{{\operatorname {id}}}

\def\r{\rightarrow}

\newtheorem{lemma}{Lemma}[section]
\newtheorem{proposition}[lemma]{Proposition}
\newtheorem{theorem}[lemma]{Theorem}
\newtheorem{corollary}[lemma]{Corollary}

\theoremstyle{definition}

\newtheorem{definition}[lemma]{Definition}

{

}

\theoremstyle{remark}

\newtheorem{remark}[lemma]{Remark}

\newdimen\uboxsep \uboxsep=1ex
\def\uboxn#1{\vtop to 0pt{\hrule height 0pt depth 0pt\vskip\uboxsep
\hbox to 0pt{\hss #1\hss}\vss}}

\def\uboxs#1{\vbox to 0pt{\vss\hbox to 0pt{\hss #1\hss}
\vskip\uboxsep\hrule height 0pt depth 0pt}}

\def\HH{\operatorname{HH}}

\def\Ob{\operatorname{Ob}}
\def\Tw{\operatorname{Tw}}
\def\aa{\mathfrak{a}}
\def\bb{\mathfrak{b}}
\def\cc{\mathfrak{c}}
\def\dd{\mathfrak{d}}

\def\Ab{\mathbf{Ab}}

\let\oldmarginpar\marginpar
\long\def\marginpar#1{\oldmarginpar{\raggedright\parskip 2pt\tiny\baselineskip=0pt\lineskip=0pt \lineskiplimit=0pt  #1}}
\def\cone{\operatorname{cone}}

\def\Free{\operatorname{Free}}

\title{A $\mathrm{k}$-linear triangulated category without a model}

\def\coDer{\operatorname{coDer}}

\def\rep{\operatorname{rep}}
\author{Alice Rizzardo}
\email[Alice Rizzardo]{alice.rizzardo@liverpool.ac.uk}
\address{Department of Mathematical Sciences\\
University of Liverpool\\
Mathematical Sciences Building\\
Liverpool L69 7ZL\\
United Kingdom}
\author{Michel Van den Bergh}
\email[Michel Van den Bergh]{michel.vandenbergh@uhasselt.be}
\address{Universiteit Hasselt\\Martelarenlaan 42\\3500 Hasselt\\Belgium}
\thanks{The first author is a Lecturer at the University of Liverpool. She is supported by EPSRC grant EP/N021649/1. The second author is a senior researcher at the Research Foundation - Flanders (FWO). He is supported by the FWO-grant G0D8616N ``Hochschild cohomology and deformation theory of triangulated categories.''}
\keywords{Triangulated category, model, enhancement}
\subjclass{13D09, 18E30, 14A22}
\begin{document}
\begin{abstract}
  In this paper we give an example of a triangulated category, linear over a field of characteristic zero, which does not carry a DG-enhancement. The only previous examples of triangulated categories without a model have been constructed by Muro, Schwede and Strickland. These examples are however not linear over a field.
\end{abstract}
\maketitle
\section{Introduction}
\subsection{Main result}
The only known examples of triangulated categories without model (not even topological) are given in \cite{muro2}.
The examples in loc.\ cit.\  are not linear over a field and furthermore they  depend on some special properties of the number $2$. In particular they satisfy $2\neq 0$ but $4=0$.

In this paper we discuss triangulated categories over a field $k$ of
characteristic zero.\footnote{Notwithstanding what we say here, almost everything we do is valid in arbitrary characteristic. However in finite characteristic we would
 also have to consider topological enhancements and we do not discuss these in the current paper.}  In this case the appropriate notion of a model is
a DG-enhancement \cite{Bondal5,CanonacoStellari,OrlovLunts}, or what
amounts to the same thing\footnote{We can always transform an $A_\infty$-enhancement into a DG-enhancement by taking its DG-hull. See \cite[p127]{Lefevre} or \cite[Appendix C]{RizzardoVdB2}.}: an $A_\infty$-enhancement (see \S\ref{sec:nomodel}). Our main result is an example of a \emph{$k$-linear triangulated category which does not carry an
$A_\infty$-enhancement}. This in particular answers positively what is described as a challenging question in the survey \cite{CanonacoStellari} by Canonaco and Stellari, namely Question 3.8. Our example also provides a negative answer to Question 3.3 of their survey.

\medskip

To describe the example we have to introduce some notation. Fix a natural number $n\ge 3$ and
let $k$ be either a field of characteristic zero or an infinite field of characteristic $>n$.
Let $R=k[x_1,\ldots,x_n]$ and let $K$ be the quotient field of~$R$. 
Furthermore let $R[\varepsilon]$ be the $R$-linear DG-algebra with $|\varepsilon|=-n+2$, $\varepsilon^2=0$, $d\varepsilon=0$. 
Let $C(R,R)$ be the Hochschild complex of $R$ and  let $\HH^n(R,R)=H^n(C(R,R))$. Let $T^n_{R/k}=\wedge^n_R\Der_k(R,R)$.
The HKR theorem furnishes an inclusion $T^n_{R/k}\subset Z^n C(R,R)$ which induces an isomorphism $T^n_{R/k}\cong \HH^n(R,R)$.
For $\eta\in T^n_{R/k}$   we let $R_\eta$ be the $k[\varepsilon]$-linear
$A_\infty$-deformation of
$R[\varepsilon]$  whose only non-trivial higher multiplication is given by $\varepsilon\eta$.
\begin{theorem}[see \S\ref{sec:noenhance}] \label{intro:thm0} Assume $n\ge 14$ and $\eta\neq 0$. Then there exists a triangulated category without $A_\infty$-en\-hancement with semi-orthogonal decomposition
$\langle D(K),D(R_\eta)\rangle$.
\end{theorem}
In the next few sections we discuss in more detail the ingredients that go into the construction of this example.
\subsection{Pre-triangulated $A_n$-categories}
An $A_\infty$-category \cite{Lefevre} is a DG-graph equipped with higher
compositions $(m_i)_{i}$ which
satisfy certain natural quadratic relations.
If only $m_i$ with
$i\le n$ are defined then we obtain the corresponding notion
of an $A_n$-category. As a general principle, for any
$A_\infty$-notion there is a corresponding $A_n$-notion in which we
consider only operations with $\le n$ arguments and we require the
axioms to only hold for expressions with $\le n$ arguments. Facts
about $A_\infty$-categories remain valid for $A_n$-categories as long
as they only involve such expressions. It is useful to note that if $\aa$ is an $A_n$-category
for $n\ge 3$ then its ``homotopy category'' $H^0(\aa)$ is an honest category.

\medskip

A DG-category is an $A_\infty$-category with $m_i=0$ for $i>2$. In their seminal paper \cite{Bondal5} Bondal and Kapranov
introduced \emph{pre-triangulated DG-categories} which, in particular, have the property that their
homotopy category is canonically
 triangulated.
Their most striking insight  is that, whereas a triangulated category is an additive category
with extra structure, a pre-triangulated DG-category is a DG-category with extra properties.

\medskip

It is well understood how to define  the analogous notion of a \emph{pre-triangulated $A_\infty$-category} (see \cite{BLM}). An $A_\infty$-category
is pre-triangulated if the natural functor $\aa\r \Tw\aa$ is a quasi-equivalence, where $\Tw\aa$ is the category of twisted complexes over $\aa$. It is easy
to see that this is equivalent to  $\aa$ being closed under suspensions, desuspensions and cones of closed maps, up to isomorphism in $H^0(\Tw \aa)$. Stating these properties
explicitly requires only a finite number of higher operations on $\aa$ and so they make sense for $A_n$-categories for $n\gg 0$.

For any $A_\infty$-category $\aa$, $H^0(\Tw\aa)$ is canonically triangulated and hence if $\aa$ is pre-triangulated then\footnote{In this introduction we will follow tradition by viewing a triangulated category as an additive category. However in the main body of the paper
we will equip a triangulated category with its  canonical graded enrichement. This means in particular that we use $H^\ast(\aa)$ rather than $H^0(\aa)$. See \S\ref{sec:notconv} for the rationale for this choice.}
 $H^0(\aa)$ is also canonically triangulated. Now it is intuitively clear that 
it should be possible to prove this using only 
a finite number of the higher operations on $\aa$. It then follows that it must be possible to define for $n\gg 0$
a notion of a pre-triangulated $A_n$-category which induces a canonical triangulation on its homotopy category.

\medskip

Unfortunately, carrying out this program naively using explicit equations seems to be a nightmare. Therefore we are forced carry over some more advanced technology from the 
$A_\infty$-context. This is done in \S\ref{sec:prelimAn}, \S\ref{sec:truncated}.
The main difficulty we face is that the definition of $\Tw\aa$ depends on higher compositions in $\aa$ of unbounded arity and therefore does not generalize to $A_n$-categories. Luckily this issue can be solved by considering twisted complexes of uniformly bounded
length. In fact we only need $\Tw_{\le 1}\aa$, which consists of twisted complexes of length two. This leads to our first main result.
\begin{theorem}[Lemma \ref{An-ness}, Definition \ref{triangulatedstructure}, Theorem \ref{th:mainth}] \label{thm:intro1}
If $\aa$ is an $A_n$-category then $\Tw_{\le 1}\aa$ is an $A_{\lfloor (n-1)/2\rfloor}$-category. If $n\ge 7$ then we say that $\aa$ is \emph{pre-triangulated} if $H^\ast(\aa)\r H^\ast(\Tw_{\le 1}\aa)$ 
is a graded equivalence.
If $\aa$ is pre-triangulated and $n\ge 13$ then $H^0(\aa)$ is canonically triangulated.
\end{theorem}
The number 13 seems quite high and we are rather curious if it can be reduced.
\subsection{Gluing}
\label{intro:gluing}
We have already pointed out that if $\aa$ is an $A_n$-category then its ``pre-triangulated hull'' $\Tw \aa$ is not well-defined. So while we have a satisfactory theory of pre-triangulated $A_n$-categories, it is unclear how to actually construct non-trivial examples of them. Luckily there is one approach which works very well. 
It turns out that pre-triangulated $A_n$-categories  admit a ``gluing'' procedure and starting from pre-triangulated $A_\infty$-categories we can in this way produce  pre-triangulated $A_n$-categories which are not themselves $A_\infty$-categories.

Let us first review gluing in the context of triangulated categories. If $\Ascr$, $\Bscr$ are triangulated categories and $\Mscr$ is a $\Bscr-\Ascr$-bimodule (an additive bifunctor $\Ascr^\circ\times\Bscr\r \Ab$)
then a \emph{gluing} of $\Ascr$, $\Bscr$ across $\Mscr$ is a triangulated category $\Cscr$ together with a semi-orthogonal decomposition $\Cscr=\langle \Ascr,\Bscr\rangle$
such that $\Cscr(A,B)=\Mscr(A,B)$ for $A\in\Ob(\Ascr)$, $B\in\Ob(\Bscr)$. The data $(\Ascr,\Bscr,\Mscr)$
determines the objects of $\Cscr$ up to isomorphism and there is a long exact sequence relating the $\Hom$-spaces in $\Cscr$ to those in $\Ascr$, $\Bscr$ and the elements of $\Mscr$. However this is as far as it goes.
Triangulated categories are too flabby to allow
one to fully construct $\Cscr$ from the triple $(\Ascr,\Bscr,\Mscr)$.

\medskip

On the other hand if $\aa$, $\bb$ are $A_\infty$-categories and $M$ is
an $A_\infty$-$\bb$-$\aa$-module then it is a routine matter to define an
$A_\infty$-gluing category $\cc=\aa\coprod_M\bb$ such that if
$\aa$, $\bb$
are pre-triangulated then so is $\cc$ and there is a
semi-orthogonal decomposition
$H^0(\cc)=\langle H^0(\aa),H^0(\bb)\rangle$ with associated bimodule
$H^0(M)$.

To prove that $\cc$ is pre-triangulated we have to prove it is closed under cones of closed maps and again it is clear that this will only involve a finite number of higher operations. Hence
the theory can be developed for $A_n$-categories. This leads to our next main result.
\begin{theorem}[Theorem \ref{th:maingluing}] \label{thm:intro2}
Assume that $n\ge 13$ and that $\aa$, $\bb$ are pretriangulated $A_n$-categories and that $M$ is an $A_n$-$\bb$-$\aa$-bimodule. Then $\aa\coprod_M\bb$ is a pre-triangulated
$A_{n-1}$ category. If $n\ge 14$, so that $H^0(\aa\coprod_M\bb)$ is triangulated by Theorem \ref{thm:intro1}, then we have a semi-orthogonal decomposition
$H^0(\aa\coprod_M\bb)=\langle H^0(\aa),H^0(\bb)\rangle$ whose associated bimodule is $H^0(M)$.
\end{theorem} 
\subsection{The counterexample}
The counterexample we describe in Theorem \ref{intro:thm0} will be more specifically of the form $\Dscr=H^0(\aa\coprod_M\bb)$ where $\aa$, $\bb$ are pre-triangulated $A_\infty$-categories and $M$ is an $A_n$-$\bb$-$\aa$-bimodule. 
We will in fact assume that $M$ is obtained from an $A_{n-1}$-functor $F:\aa\r \bb$ via $M(A,B)=\bb(FA,B)$.
By
Theorems \ref{thm:intro1} and \ref{thm:intro2}, $\Dscr$ is canonically triangulated for $n\gg 0$.  Moreover any $A_\infty$-enhancement on $\Dscr$ induces $A_\infty$-enhancements $\aa'$, $\bb'$
on $H^0(\aa)$, $H^0(\bb)$ as well as an $A_\infty$-functor $F':\aa'\r \bb'$ such that $H^0(F')=H^0(F)$. One may hope to be able to prove that such $F'$ does not exist. This then implies that an
$A_\infty$-enhancement on $\Dscr$ does not exist.

We carry out this program with $\aa$, $\bb$ being the standard $A_\infty$-enhancements of $D(K)$, $D(R_\eta)$ for $\eta\neq 0$ (see \S\ref{sec:nomodel}).
The exact functor
\[
f:D(K)\r D(R_\eta):K\r K_\eta
\]
(defined using the fact that $D(K)$ is the category of graded $K$-vector spaces, equipped with its unique triangulation) lifts to an $A_{n-1}$-functor $F$: by \cite[Lemma 7.2.1]{RizzardoVdB2} this follows from the fact that $H^i(K_\eta)=0$ for $i=0,\ldots,-n+3$.
However, using the fact that $\eta\neq 0$ one deduces that $f$ does not lift to an $A_\infty$-functor, even if
we are allowed to change enhancements. This follows from the fact that the enhancement on $D(R_\eta)$ is actually unique in a weak, but sufficient, sense. This is proved using higher
Toda brackets (see Proposition \ref{lem:enhancement}). 
This finishes the proof that an $A_\infty$-enhancement on $\Dscr$ does not exist. 

\section{Acknowledgement} The authors thank Alexey Bondal and Dmitri Orlov for several interesting discussions around
the possibility of gluing a non-enhanceable functor to obtain a triangulated category without model.
\section{Notation and conventions}
\label{sec:notconv}
Below $k$ is an arbitrary field, except in \S\ref{sec:noenhance} where it will be subject to some restrictions. Unless otherwise specified, categories are pre-additive (enriched in abelian groups),
except when we are in an $A_n$-context. In that case we assume all objects and constructions are $k$-linear. 

Triangulated categories will be equipped with their canonical graded enhancement (see \S\ref{sec:trianggraded}).
The motivation for this is that the principal ``homotopy invariant'' associated to an $A_n$-category $\aa$ is $H^\ast(\aa)$ as $H^0(\aa)$ loses too much information in general. If
$\aa$ is pre-triangulated then $H^\ast(\aa)$ can be recovered from $H^0(\aa)$ together with a ``shift functor'' but, since the shift functor is not canonical (despite being unique up
to unique isomorphism), this extra step creates some complications, notably with signs, which are often unnecessary.
In any case, not all $A_n$-categories we will encounter will be pre-triangulated.

In situations where the shift functor is canonical we will use it. The most obvious case is graded objects over an abelian category $\Ascr$. If $A^\bullet=(A_i)_{i\in \ZZ}$ is such an object then
we put $\Sigma^n (A^\bullet)_i=A_{i+n}$. If $f:A^\bullet\r B^\bullet$ has degree $i$ then we put $\Sigma^n f =(-1)^{ni} f$. If $A^\bullet$ is a graded object over $\Ab$ and $x\in A_i$ then we write
$sx$ for $x$ considered as an element of $(\Sigma A^\bullet)_{i-1}$. The ``degree change operator'' $s$ makes it easy to find the correct sign in formulas using the Koszul convention.
\section{Preliminaries on triangulated categories}
\subsection{Graded categories}
\label{sec:notes}
For us  a \emph{graded category} is a category enriched in $\ZZ$-graded abelian groups.
Assume that $\aa$ is a graded category and let $X\in \Ob(\aa)$. A \emph{suspension} of $X$ is a pair $(Y,\eta)$ where $Y\in \Ob(\aa)$ and $\eta\in \aa(X,Y)_{-1}$ is invertible. Conversely
we call $(X,\eta)$ a \emph{desuspension} of $Y$. (De)suspensions are clearly functorial if they exist. So if  every object $X$ has a suspension $(Y,\eta)$ we may define a functor
$\Sigma:\aa\r \aa$ by putting $\Sigma X=Y$ and requiring for maps $f\in \aa(X,X')$ that the following diagram
\[
\xymatrix{
X\ar[r]^{\eta}\ar[d]_{f} & \Sigma X\ar[d]^{\Sigma f}\\
X'\ar[r]_{\eta} & \Sigma X'\\
}
\]
commutes up to a sign $(-1)^{|f|}$. It is clear that $\Sigma$ is unique up to unique equivalence.
We say that $\aa$ has a \emph{shift functor} $\Sigma$ if every object has a suspension and a desuspension and $\Sigma$ is as above. In this case $\Sigma$ is an auto-equivalence.
\subsection{Graded categories from pre-additive categories with shift functor}
\label{sec:linear}
Now assume that $\aa$ is a \emph{pre-additive category} (i.e.\ a category enriched in abelian groups) equipped with an auto-equivalence $\Sigma$. Then we can make $\aa$ into a graded category $\tilde{\aa}$ with the same objects 
by putting for $n\in \ZZ$
\[
\tilde{\aa}(A,B)_n:=\aa(A,\Sigma^n B)
\]
and with compositions
\[
\tilde{\aa}(B,C)_m\times \tilde{\aa}(A,B)_n\r \tilde{\aa}(A,C)_{m+n}:(g,f)\mapsto (-1)^{nm}\Sigma^n g\circ f
\]
We obtain that $\Sigma$ is a shift functor on $\tilde{\aa}$ in the sense of \S\ref{sec:notes}. 
\subsection{Triangles}
\label{sec:triangle}
A \emph{triangle} in a graded category $\aa$ is a diagram
\[
\xymatrix{
&C\ar[dl]_{h}^{(1)}&\\
A\ar[rr]_f&&B \ar[ul]_g
}
\]
with $A,B,C\in \Ob(\aa)$ and $|f|=|g|=0$, $|h|=1$. To save space a triangle will usually be written in linear form
\[
A\xrightarrow{f} B\xrightarrow{g} C\xrightarrow[(1)]{h}  A.
\]
If $\aa$ is equipped with a shift functor then a triangle can also be written in ``traditional'' form
\[
A\xrightarrow{f} B\xrightarrow{g} C\xrightarrow{h} \Sigma A.
\]
A morphism of triangles is given by three degree zero morphisms fitting into the obvious commutative diagram.
\subsection{Triangulated categories as graded categories}
\label{sec:trianggraded}
We will assume that the reader is familiar with the standard axioms for triangulated categories \cite{Verdier}.
If $(\Tscr,\Sigma)$ is triangulated category in the traditional sense 
then it can be made into a graded category as in \S\ref{sec:linear}. In this section we will reformulate
the usual axioms of triangulated categories in such a way that they
do not explicitly refer to a shift functor.
\begin{definition}
A \emph{triangulated} category $\Tscr$ is a graded  category equipped
with a collection of ``distinguished'' triangles
such that\footnote{Morphisms in a graded category whose degree is not specified are assumed to have degree zero. This convention is maintained throughout this document.}
\begin{enumerate}
\item[TR0] $\Tscr$ admits (possibly empty) finite direct sums and every object has a suspension and a desuspension.
\item[TR1] 
\begin{itemize}
\item
For any object $X\in \Ob(\Tscr)$ the following triangle is distinguished:
\[
 X\xrightarrow{\Id_X} X\xrightarrow{0_0} 0\xrightarrow[(1)]{0_1} X
\]
where $0$ is a zero object (which exists by TR0) and where $0_i$ is the zero morphism in $\Tscr(U,V)_i$.
\item
For any morphism $u : X\r Y$ in $\Tscr$ of degree zero, there is an object $Z$ (called a mapping cone of the morphism $u$) fitting into a distinguished triangle
\[
 X{\xrightarrow {u}}Y\xrightarrow{} Z\xrightarrow[(1)]{} X
\]
\item
Any triangle isomorphic to a distinguished triangle is distinguished.
\end{itemize}
\item[TR2]
If
\[
\xymatrix{
&Z\ar[dl]_{w}^{(1)}&\\
X\ar[rr]_u&&Y \ar[ul]_v
}
\]
is a distinguished triangle
then so are the two ``rotated triangles''
\[
\xymatrix{
&Z\ar[dl]_{\eta w}&\\
X'\ar[rr]^{(1)}_{-u\eta^{-1}}&&Y \ar[ul]_v
}
\]
\[
\xymatrix{
&Z'\ar[dl]_{-w\gamma}&\\
X\ar[rr]_u&&Y \ar[ul]_{\gamma^{-1}v}^{(1)}
}
\]
where $X\xrightarrow{\eta} X'$ is a suspension of $X$ and $Z'\xrightarrow{\gamma} Z$ is a desuspension of $Z$.
\item[TR3] A commutative diagram of solid arrows
\[
\xymatrix{
 X\ar[d]\ar[r]&Y\ar[d]\ar[r]&Z\ar@{.>}[d]\ar[r]^{(1)}& X\ar[d]\\ 
 X'\ar[r]&Y'\ar[r]&Z'\ar[r]_{(1)}& X' 
}
\]
in which the rows are distinguished can be completed with the dotted arrow.
\item[TR4]
For every upper cap of an octahedron (drawn on the left) there is a corresponding lower cap (drawn on the right).
\begin{equation}
\label{eq:octahedron1}
\xymatrix{
Z\ar[rrrr]&&&&{X'}\ar[ddll]_{(1)}\ar[dddd]^{(1)}\\
&&\bold{d}&&\\
&\circlearrowleft&Y\ar[lluu]\ar[ddrr]&\circlearrowleft&\\
&&\bold{d}&&\\
X\ar[uuuu]\ar[rruu]&&&&{Z'}\ar[llll]^{(1)}
}
\xymatrix{
Z\ar[rrrr]\ar@{.>}[ddrr]&&&&{X'}\ar[dddd]^{(1)}\\
&&\circlearrowleft&&\\
&\bold{d}&\textcolor{gray}{Y'}\ar@{.>}[rruu]\ar@{.>}[ddll]_{(1)}&\bold{d}&\\
&&\circlearrowleft&&\\
X\ar[uuuu]&&&&{Z'}\ar[llll]^{(1)}\ar@{.>}[lluu]
}
\end{equation}
such that in addition the compositions
$Y\r Z\r Y'$ and $Y\r Z'\r Y'$ are the same and similarly the  compositions $Y'\r X\r Y$ and $Y'\r X'\r Y$ are the same.
In the diagram the triangles marked $\bold{d}$ are distinguished and those marked with
$\circlearrowleft$  are commutative,
\end{enumerate}
\end{definition}
\section{Preliminaries about $A_n$-categories}
\label{sec:prelimAn}
Let $n\ge 0$. As a general principle, for any $A_\infty$-notion there is a corresponding $A_n$-notion in which we
consider only operations with $\le n$ arguments and we  
require the axioms to only hold for expressions with $\le n$ arguments. Facts about $A_\infty$-categories
remain valid for $A_n$-categories as long as they only involve such expressions.
We discuss this below. Throughout we
place ourselves in the strictly unital context.
\subsection{$A_n$-categories and functors}
\label{sec:Ancatsandfunctors}
\begin{definition}[\cite{Lefevre}]
An $A_n$-category $\aa$ is the data of:
\begin{itemize}
\item A set of objects $\Ob(\aa)$.
\item For each couple $(A,A')$ of objects of $\aa$, a graded vector space of morphisms $\aa(A,A')$.
We call $\aa(A,A')$ the $\Hom$-space between $A$ and $A'$. A (homogeneous) element of $\aa(A,A')$ is called a morphism
(or sometimes an arrow).
\item For each sequence $(A_0,\ldots,A_i)$ of objects of $\aa$ with $1\le i\le n$, ``higher'' compositions
\[
b_i:\Sigma\aa(A_{i-1},A_i)\otimes \ldots \otimes \Sigma\aa(A_0,A_1)\r \Sigma \aa(A_0,A_i)
\]
of degree 1 verifying $(*)_i$ of \cite[definition 1.2.1.1]{Lefevre}.
\item For each object $A$ an \emph{identity (or unit) element} $\Id_A\in \aa(A,A)_0$ satisfying
\begin{align*}
b_i(\dots,s\Id_A,\dots)&=0&&\text{(for $i=1$ and $3\le i\le n$)}\\
b_2(sf,s\Id_A)&=(-1)^{|f|}sf&&\text{if $n\ge 2$}\\
b_2(s\Id_A,sg)&=sg&&\text{if $n\ge 2$}
\end{align*}
\end{itemize}
\end{definition}
If the identities hold for every $i$ then we get the notion of an 
$A_{\infty}$-category. 
Below
an $A_n$-category will be silently considered as an $A_m$-category for
all $m\le n$.

As for $A_\infty$-categories is it sometimes more convenient to
express the higher compositions as operations
\[
m_i:\aa(A_{i-1},A_i)\otimes \ldots \otimes \aa(A_0,A_1)\r  \aa(A_0,A_i)
\]
of degree $2-n$ where $(m_n)_n$ and $(b_n)_n$ are related by
 $b_n=s^{-n+1}m_n$ 
so that in particular using the Koszul convention we obtain
\begin{equation}
\label{eq:signs}
\begin{aligned}
b_1(sf)&=-sm_1(f)\\
b_2(sg,sf)&=(-1)^{|g|}sm_2(g,f)\\
b_3(sh,sg,sf)&=(-1)^{|g|+1} s m_3(h,g,f)
\end{aligned}
\end{equation}
Sometimes we write $d\!f=m_1(f)$ and $gf=m_2(g,f)$.
It is useful to consider the case of low $n$. 
\begin{enumerate}
\item An $A_0$-category is simply a directed graph (with distinguished ``identity arrows'') whose $\Hom$-spaces are graded vector spaces. We call this a \emph{graded graph}. 
\item An $A_1$-category is a graded graph whose arrows form complexes of vector spaces (the differential is given by $m_1$ and it annihilates identity arrows). 
We call this a \emph{DG-graph}. A DG-graph $\aa$ has an associated graded graph
$H^\ast(\aa)$ obtained by replacing the $\Hom$-spaces in $\aa$ by their cohomology. 
A morphism $f$ in $\aa$ is called \emph{closed} if $m_1(f)=0$. We denote by $Z^0\aa$ the $k$-linear graph which has the same objects as $\aa$ and whose morphisms are the closed morphisms of degree zero.
\item An $A_2$ category is a DG-graph equipped with a bilinear
  composition of arrows given by $m_2$ (for which the identity arrows behave as unit elements) which is compatible with $m_1$. In 
particular $m_2$ descends to well-defined operations on $H^\ast(\aa)$ and $Z^0\aa$.
\item For $n\ge 3$ the composition on $H^\ast(\aa)$ induced by $m_2$ is associative and hence in particular $H^\ast(\aa)$ is a graded category.
\end{enumerate}

\begin{definition}\label{A_n functors}
An $A_n$-functor $f:\aa\to \bb$ between two $A_n$-categories $\aa$ and $\bb$  is the data of
\begin{itemize}
\item A map on objects $f:\Ob(\aa) \to \Ob(\bb)$.
\item For each sequence $(A_0,\ldots,A_i)$ of objects of $\aa$ with $i\leq n$, compositions
\[
f_i:\Sigma\aa(A_{i-1},A_i)\otimes \ldots \otimes \Sigma\aa(A_0,A_1) \to \Sigma\bb(f(A_0),f(A_i))
\]
of degree zero verifying $(**)_i$ of \cite[definition 1.2.1.2]{Lefevre} for $i=1,\ldots,n$.
\item If $n\ge 1$ then for each $A\in \Ob(\aa)$ we have $f_1(s\Id_A)=s\Id_{f(A)}$ and $f_n(\ldots,s\Id_A,\ldots)=0$ for $n\ge 2$.
\end{itemize}
\end{definition}
Again it is instructive to unravel this definition for small values of $n$. 
\begin{enumerate}
\item An $A_0$-functor is just a map between sets of objects (there is no compatibility with morphisms). 
\item An $A_1$-functor $f:\aa\r \bb$ is a morphism of DG-graphs.
In particular we have an induced morphism of graded graphs $H^\ast(f):=H^\ast(f_1)$.
\item If $f$ is an $A_n$-functor for $n\ge 2$ then $H^\ast(f)$ is compatible
with compositions. 
In particular, if $f$ is an $A_2$-functor between $A_3$-categories 
then $H^\ast(f)$ is a graded functor.
\end{enumerate}
Like $A_\infty$-notions one may also approach $A_n$-notions via cocategories. Let $\aa$ be a graded graph. Then $(B\aa)_{\le n}$ is the graded cocategory with $\Hom$-spaces
\begin{equation}
\label{eq:bar}
\begin{aligned}
(B\aa)_{\le n}(A,B)&=\bigoplus_{i=1}^{n} (\Sigma \aa)^{\otimes i}(A,B)\\
(\Sigma \aa)^{\otimes i}(A,B)&=\bigoplus_{A=A_0,\ldots,A_{i}=B} \Sigma\aa(A_{i-1},A_i)\otimes \ldots \otimes \Sigma\aa(A_0,A_1)
\end{aligned}
\end{equation}
equipped with the usual \emph{bar coproduct}. I.e. if $(sf_{i-1}|\cdots|sf_0):=
sf_{i-1}\otimes\cdots\otimes sf_0\in (\Sigma \aa)^{\otimes i}$ then
\[
\Delta(sf_{i-1}|\cdots|sf_0)=\sum_{j=1}^{i-1} (sf_{i-1}|\cdots |sf_j)\otimes (sf_{j-1}|\cdots |sf_0)
\]
If we ignore the compatibility with units then  an $A_n$-structure on $\aa$ is the same as a \emph{codifferential} on  $(B\aa)_{\le n}$, i.e. a coderivation~$b$ of degree one satisfying
$b\circ b=0$.  Similarly, ignoring units, an $A_n$-functor $f:\aa \r \bb$ is the same as a cofunctor
$(B\aa)_{\le n}\r (B\bb)_{\le n}$ commuting with the codifferentials on $(B\aa)_{\le n}$ and $(B\bb)_{\le n}$. With this observation one may define the composition of $A_n$-functors simply as the composition of the corresponding cofunctors.
\subsection{Some auxilliary definitions}
\begin{definition} \label{def:auxilliary}
Let $f:\aa\r \bb$ be an $A_m$-functor between $A_n$-categories, for $m\le n$. Then
\begin{enumerate}
\item $f$ is \emph{strict} provided $m\ge 1$ and $f_i=0$ for $i\ge 2$. Equivalently $f_1$ commutes with higher compositions
with arity at most $m$.
\item $f$ is \emph{fully faithful} if it is strict and for all $A,A'\in \Ob(\aa)$ we have that
$
\aa(A,A')\r \bb(fA,fA')
$
is an isomorphism of graded vector spaces.
\item $f$ is a \emph{quasi-fully faithful} if $m\ge 2$, $n\ge 3$ and $H^\ast(f):H^\ast(\aa)\r H^\ast(\bb)$ is fully faithful.
\item $f$ is a \emph{quasi-isomorphism} if $m\ge 2$, $n\ge 3$ and $H^\ast(f):H^\ast(\aa)\r H^\ast(\bb)$ is an isomorphism.
\item $f$ is a \emph{quasi-equivalence} if $m\ge 2$, $n\ge 3$ and $H^\ast(f):H^\ast(\aa)\r H^\ast(\bb)$ is an equivalence. 
\end{enumerate}
\end{definition}
\subsection{The category of functors between $A_n$-categories} 
Here we discuss some concepts from \cite[Chapter 8]{Lefevre}. 
As indicated above, the (decomposable) arrows of $(B\aa)_{\le n}$ are usually written as $(sf_{i-1}|\cdots|sf_0)$ for a path of 
$1\le i\le n$ composable arrows $f_0,\ldots,f_{i-1}$ in $\aa$.
We let $(B^+\aa)_{\le n}$ be the coaugmented cocategory obtained  by also admitting empty paths $()_A$ 
starting and ending  in $A\in \Ob(\aa)$ (see \S2.1.2 in loc. cit.).
More precisely we have
\[
(B^+\aa)_{\le n}(A,B)=
\begin{cases}
(B\aa)_{\le n}(A,B)&\text{if $A\neq B$}\\
k()_A\oplus (B\aa)_{\le n}(A,A)&\text{if $A=B$}.\\
\end{cases}
\]
with $|()_A|=0$.
The coproduct $\Delta^+(t)$ for $t\in (B\aa)_{\le n}(A,B)$  is defined as 
\[
\Delta^+(t)=()_B\otimes t+t\otimes ()_A+\Delta(t),
\]
where $\Delta$ is the coproduct on $(B\aa)_{\le n}$ and furthermore
$
\Delta^+(()_A)=()_A\otimes ()_A
$.
If $(B\aa)_{\le n}$ is equipped with a codifferential $b$ then we extend it to $(B^+\aa)_{\le n}$
by putting $b(()_A)=0\in \aa(A,A)$.

Given two $A_n$-categories $\cc$ and $\dd$, 
denote by $A_n(\cc,\dd)$ the set of $A_n$-functors $\cc\to \dd$. Now
assume that $\aa$, $\bb$ are respectively $A_m$, $A_n$-categories for
$m\le n-1$.  We will equip $A_m(\aa,\bb)$ with the structure of an
$A_{n-m}$-category as follows:

\begin{definition}[Morphisms in $A_m(\aa,\bb)$]\label{A_n morphisms} 
Assume $m\le n-1$. Let $f_1,f_2:\aa\r \bb$ be $A_m$-functors. We view these as cofunctors $(B^+\aa)_{\le m}\r B\bb_{\le n}$ by putting $f_{i,0}()_A:=f_{i,0}(()_A)=0$.
 Then  
\begin{align*}
\Sigma\Hom(f_1,f_2)=\{h\in\coDer_{f_1,f_2}(B^+\aa_{\leq m},B\bb_{\le n})\mid \forall A\in \Ob(\aa):h(\cdots\otimes s\Id_A\otimes\cdots)=0\}
\end{align*}
\end{definition}
Here $\coDer_{f_1,f_2}((B^+\aa)_{\leq m},B\bb_{\le n})$ consists of collections  $k$-linear morphisms $h(A,A'):B^+\aa_{\le m}(A,A')\r B\bb_{\le n}(f_1(A),f_2(A'))$
such that $h=h(A,A')_{A,A'}$
satisfies the following identity for $u\in (B^+\aa)_{\le m}$ 
\[
\Delta (h(u))=\sum_{(u)} (f_2\otimes h+h\otimes f_1)(u_{(1)}\otimes u_{(2)}).
\]
where (using the Sweedler notation) $\Delta^+(u)=\sum_{(u)} u_{(1)}\otimes u_{(2)}$. It follows that $h\in \Sigma\Hom(f_1.f_2)$ is determined by the ``Taylor coefficients''
\begin{equation}
\label{eq:homotopy}
h_k:\Sigma \aa(A_{k-1},A_k)\otimes \Sigma\aa(A_{k-2},A_{k-1})\otimes\cdots
\otimes \Sigma\aa(A_0,A_1)\r \Sigma\bb(f_1(A_0),f_2(A_k))
\end{equation}
for $1\le k\le m$ as well as for each $A\in \Ob(\aa)$ an element $h()_A:=h_0(()_A)\in \Sigma\bb(f_1(A),f_2(A))$ and the corresponding coderivation is given by
\begin{equation}
\label{eq:hformula}
h=\sum_{\sum_{t=1}^q j_t+\sum_{s=1}^p i_s+k\le m} f_{2,j_q}\otimes \cdots\otimes f_{2,j_1}\otimes h_k\otimes f_{1,i_p}\otimes\cdots \otimes f_{1,i_1}
\end{equation}
where the right-hand side is restricted to terms which have $\le m$ arguments. Note that $h$ sends $(B^+\aa)_{\le m}$ to $(B\bb)_{\le m+1}$ (as the $f$'s take at least one
argument but $h_0$ takes zero arguments). So
since $m\le n-1$, $h$ is indeed well defined.
\begin{definition}[The differential on $A_m(\aa,\bb)$] If $m\le n-1$ and $h\in \Sigma\Hom(f_1,f_2)$
then $b_1(h)=[b,h]=b\circ h-(-1)^{|h|} h\circ b$. Concretely
\begin{multline}
\label{eq:higher2}
b_1({{h}})_k
=\sum_{\sum_{t=1}^q j_t+\sum_{s=1}^p i_s+l=k} b_{p+q+1}\circ(f_{2,j_q}\otimes\cdots\otimes f_{2,j_1}
 \otimes {h_l}\otimes f_{1,i_p}\otimes\cdots\otimes f_{1,i_1})\\-(-1)^{|{{h}}|}\sum_{a_0+a_1+l=k} {{h}}_{1+a_0+a_1}\circ (\Id^{\otimes a_0}\otimes b_l\otimes \Id^{\otimes a_1})
\end{multline}
\end{definition}
\begin{definition}[The higher multiplications on $A_m(\aa,\bb)$]
Assume we have morphisms
\[
f_0\xrightarrow{{{h}}_1}f_1\xrightarrow{{{h}}_2}\cdots \xrightarrow{{{h}}_{k}} f_k
\]
represented by
\[
h_i\in \coDer_{f_{i-1},f_i}(B^+\aa_{\le m}, B\bb_{\le n})
\]
and assume $2\le k\le n-m$. Then we put
\begin{multline}
\label{eq:higher1}
{{h}}_{k}\cup\cdots\cup {{h}}_1=\sum f_{k,i_{k,p_k}}\otimes\cdots \otimes f_{k,i_{k1}}
 \otimes {{h}}_{k,u_k}\otimes f_{k-1,i_{k-1,p_{k-1}}}\otimes\cdots\\\cdots \otimes f_{k-1,i_{k-1,1}} \otimes\cdots\otimes  f_{1,i_{1p_1}}\otimes\cdots\otimes f_{1,i_{11}}
\otimes {{h}}_{1,u_1}\otimes  f_{0,i_{0p_0}}\otimes\cdots \otimes f_{0,i_{01}} 
\end{multline}
and $
b_k(h_k,\ldots,h_1)_l=(b\circ(h_k\cup\dots\cup h_1))_l
$
where $(-)_l$ denotes Taylor coefficients.
\end{definition}
Note that on the right-hand side of \eqref{eq:higher1} the $f$'s take at least one argument but the $h$'s may take zero arguments. 
It follows that 
${h}_{k}\cup\cdots\cup {{h}}_1$ maps $B^+\aa_{\le m}$ to $B\bb_{\le m+k}$, and hence by the hypothesis $k\le n-m$
 is a well defined element
of $\Hom(B^+\aa_{\le m}, B\bb_{\le n})$. It is however not a coderivation. 
Instead is it inductively characterized by the following property for $u\in (B^+\aa)_{\le m}$
(using again the Sweedler notation)
\begin{multline*}
\Delta(({h}_{k}\cup\cdots\cup {{h}}_1)(u))=\sum_{(u)} \biggl(f_k\otimes (h_k\cup\cdots\cup h_1)+(h_k\cup\cdots\cup h_1)\otimes f_1
\\+\sum_{1\le j\le k} 
(h_{k}\cup\cdots\cup h_{j+1})\otimes (h_{j}\cup\cdots\cup h_{1})\biggr)(u_{(1)}\otimes u_{(2)})
\end{multline*}
One checks
\begin{lemma}  \label{lem:An-m}
The collection of maps $(b_i)_{i=1,\ldots,n-m}$ makes $A_m(\aa,\bb)$ into an $A_{n-m}$-category.
\end{lemma}
\subsection{Homotopies and homotopic functors}
\label{sec:homhom}
Let $\aa$, $\bb$ be $A_n$-categories, let $1\le m\le n-1$ (thus $n\ge 2$) and let $h\in \Sigma Z^0A_m(\aa,\bb)(f_1,f_2)$. Then $h\in \coDer_{f_1,f_2}(B^+\aa_{\le m},B\bb_{\le n})_{-1}$ and $[b,h]=0$.
Let $(h_k)_{k=0,\ldots,m}$ be the Taylor coefficients of $h$. Specializing \eqref{eq:higher2} to $k=0,1$ we find that $h_0()_A\in (\Sigma \bb(f_1A,f_2A))_{-1}=\bb(f_1A,f_2A)_0$ satisfies
$d(h_0()_A)=0$ and
\[
b_1\circ h_1+b_2\circ (h_0\otimes f_1+f_2\otimes h_0)+h_1\circ b_1=0
\]
Evaluating this on $st$ for $t\in \aa(A,B)$ we find
\begin{equation}
\label{eq:h1condition}
b_1(h_1(st))+b_2(h_{0}()_B,f_1(st))+(-1)^{|t|+1}b_2(f_2(st),h_{0}()_A)+h_1(b_1(st))=0
\end{equation}
Put $h_0()_A=sh_{0,A}$. Using the usual sign convention $h_1(st)=-sh_1(t)$, etc\dots 
together with \eqref{eq:signs} this may
be rewritten as
\[
m_1(h_1(t))+m_2(h_{0,B},f_1(t))+(-1)^{|t|+1}(-1)^{|t|}m_2(f_2(t),h_{0,A})+h_1(m_1(t))=0
\]

So we find in particular that $H^\ast(h_0)$ defines a natural transformation $H^\ast(f_1)\r H^\ast(f_2)$. 
\begin{definition} Let $h,f_1,f_2$ be as above but assume $n\ge 3$. We say that $h$ is a \emph{homotopy} $h:f_1\r f_2$
if $H^*(h_0)$ is a natural isomorphism, i.e.\ if for all $A\in \Ob(\aa)$, $H^\ast(h_{0,A})\in H^\ast(\aa)(A,A)$
is invertible. We say that $f_1$, $f_2$ are \emph{homotopic} if there exists a homotopy $h:f_1\r f_2$. 
\end{definition}
\begin{lemma} \label{lem:invertible}
Assume $1\le m\le n-3$. Then $h:f_1\r f_2$ is a homotopy if and only if $H^\ast(h)$ is invertible
in $H^\ast(A_m(\aa,\bb))$ (the latter is a genuine category because of the restriction on $m,n$). In particular the relation of being homotopic is an equivalence relation.
\end{lemma}
\begin{proof} We have $(hh')_0=h_0h'_0$. So if $h$ is invertible then it is a homotopy.
Assume now $h_0$ is invertible.
Consider the morphism of complexes
\[
S:A_m(\aa,\bb)(f_2,f_1)\r A_m(\bb,\bb)(f_2,f_2):h'\mapsto m_2(h,h')
\]
Using an appropriate spectral sequence one finds that $S$ is a quasi-isomorphism. Hence
there exists $h'\in Z^0A_m(\aa,\bb)(f_2,f_1)$ such that $m_2(h,h')-\Id_{f_2}$ has zero image in $H^\ast(A_m(\aa,\bb)(f_2,f_2))$.
\end{proof}
\subsection{Inverting quasi-equivalences}
We prove some $A_n$-versions of results which are well-known in the $A_\infty$-setting (e.g.\ \cite[Th\'eor\`eme 9.2.0.4]{Lefevre}).
\begin{lemma} \label{lem:quasi-inverse}
Let $\aa$, $\bb$ be $A_n$ categories for $n\ge 3$ and let $f:\aa\r\bb$ be an $A_n$-functor
which is a quasi-equivalence.
There exists an $A_{n-1}$-quasi-equivalence $g:\bb\r \aa$ such
that $fg$
and $\Id_\bb$ are homotopic. Moreover the quasi-inverse $H^*(g_1)$ to $H^\ast(f_1)$ may be chosen freely.
\end{lemma}
\begin{proof}  If $\aa$ is an $A_n$-category then we define $\bar{\aa}$ as the DG-graph
obtained from $\aa$ by dividing out identities. I.e.\ 
\[
\bar{\aa}(A,B)=
\begin{cases}
\aa(A,B)&\text{if $A\neq B$}\\
\aa(A,A)/k\Id_A&\text{if $A=B$}
\end{cases}
\]
Note that formally $f$ is a cofunctor $B\aa_{\le n}\r B\bb_{\le n}$ such that $[b,f]=0$. Likewise $g$ should be a cofunctor $B\bb_{\le n-1}\r B\aa_{\le n-1}$ satisfying $[b,g]=0$ %\footnote{When we write something like $[b,g_{\le p}]$ we really
%mean $b\circ g_{\le p}-g_{\le p} \circ b_{\le p}$, where $b_{\le p}$ is the restriction of $b$ to $(B\bb)_{\le p}$. We hope this will always be clear from the context.}   
and the homotopy $h:fg\r \Id_\bb$  should be an element of  $\operatorname{coDer}_{fg,\Id_\bb}(B^+\bb_{\le n-1},B\bb_{\le n})$ of degree $-1$ satisfying $[b,h]=0$ such that $H^\ast(h_0)$ is a natural isomorphism
$H^\ast(fg)\r \Id_{H^\ast(\bb)}$.

We will construct $g$ and $h$ step by step. The existence of $g_1$ and $h_0,h_1$ follows simply from the fact that $f$ is a quasi-equivalence: we choose a unit preserving graded graph homomorphism
$g_1:\bb\r \aa$ commuting with differentials
such that there is a natural isomorphism $H^*(f_1)H^*(g_1)\r \Id_{H^\ast(\bb)}$. We choose $h_0:(B\bb)_0\r \bb$ in such away that this natural isomorphism is of the form $H^\ast(h_0)$
and then we choose $h_1$ such that the equation \eqref{eq:h1condition} holds.

Assume that for $1\le m<n-1$ that we have constructed a cofunctor $g_{\le m}:(B\bb)_{\le m}\r (B\bb)_{\le m}$ satisfying  
$[b,g_{\le m}]=0$ and a homotopy $h_{\le m}:fg_{\le m}\r \Id_\bb$.
We will extend the maps $(g_{\le m},h_{\le m})$ to maps $(g_{\le m+1},h_{\le m+1})$ with the same properties.

As a first approximation we extend $g_{\le m}, h_{\le m}$  to respectively a cofunctor ${}{g}_{\le m+1}:B\bb_{\le m+1}\r B\aa_{\le m+1}$ and a 
$(fg_{\le m+1},\Id_{\bb})$-coderivation ${h}_{\le m+1}:B^+\bb_{\le m+1}\r B\bb_{m+1}$  
by setting ${}{g}_{m+1},h_{m+1}:(\Sigma\bar{\bb})^{\otimes m+1}\r \Sigma\aa$ 
equal to zero (see \eqref{eq:bar} for the definition of $\Sigma\bar{\bb}^{\otimes m+1}$). Here $\pi=[b,{}{g}_{\le m+1}]$ is zero on $(\Sigma \bb)^{\otimes i}$, $i\le m$
and hence it may be regarded as a map 
$(\Sigma\bar{\bb})^{\otimes m+1}\r \Sigma\aa$. Moreover $0=[b,\pi]=[b_1,\pi]$. So $\pi$ is closed for the $b_1$-differential
and since $f\pi=[b,fg_{\le m+1}]$ is zero on cohomology and $f$ is a quasi-isomorphism, 
$\pi$ is equal to zero in cohomology as well.
 In other words, there exists
$\delta_{m+1}:\bar{\bb}^{\otimes m+1}\r \aa$ such that $\pi=[b_1,\delta_{m+1}]$. We now replace ${}{g}_{m+1}$ by ${}{g}_{m+1}-\delta_{m+1}$. Then
$[b,{}{g}_{\le m+1}]=0$. In other words ${}{g}_{\le m+1}$ is an $A_{m+1}$-morphism.

Put $D=[b,h_{\le m+1}]$ (see \eqref{eq:higher2}).
Then $D$ is a $(fg_{\le m+1},\Id_{\bb})$-coderivation $(B^+\bb)_{\le m+1}\r B\bb_{\le n}$ which is zero on $(B\bb)_{\le m}$ and hence it can be considered as a map $(\Sigma\bar{\bb})^{\otimes m+1}\r \Sigma\bb$.  
Hence we have
\begin{equation}
\label{eq:chain}
[b_1,D]=[b,D]=0
\end{equation}
We will now try to choose $\sigma_{m+1}:(\Sigma\overline{\bb})^{\otimes m+1}\r \Sigma\aa$, 
$\tau_{m+1}:(\Sigma\overline{\bb})^{\otimes m+1}\r \Sigma\bb$ 
such that for $g'_{m+1}=g_{m+1}+\sigma_{m+1}$, $h'_{m+1}=h_{m+1}+\tau_{m+1}$, $g'_i=g_i$, $h'_i=h_1$ for $i\le m$ we have $[b,g'_{\le m+1}]=0$, $[b',h'_{\le m+1}]=0$ where here $[b',-]$ is the differential \eqref{eq:higher2} computed with $f_1=fg'_{\le m+1}$ and $f_2=\Id_{\bb}$.  
The conditions we have to satisfy are
\begin{align}
\label{eq:quasi1} 0=[b,g'_{\le m+1}]&=[b_1,\sigma_{m+1}]\\
\label{eq:quasi2} 0=[b',h'_{\le m+1}]&=D+b_2\circ (h_0\otimes f_1(g_{m+1}+\sigma_{m+1}))+[b_1,\tau_{m+1}]
\end{align}
We claim these equations have a solution. First note that 
\eqref{eq:quasi2}
may be written as
\begin{equation}
\label{eq:quasi3}
b_2\circ (h_0\otimes f_1\sigma_{m+1})=-D-b_2\circ (h_0\otimes f_1 g_{m+1})\qquad \mod \im [b_1,-]
\end{equation}
Recall that here we have  $[b_1,D]=0$, $b_1\circ h_0=0$ (see \S\ref{sec:homhom}), $[b_1,f_1g_{m+1}]=0$. Hence if we have a solution $\sigma_{m+1}$ to \eqref{eq:quasi1}
\eqref{eq:quasi3} and we replace $\sigma_{m+1}$  by $\sigma_{m+1}+[b_1,s]$
then it is still a solution.

It follows we may combine \eqref{eq:quasi1}\eqref{eq:quasi3} into a single equation 
\[
\bar{b}_2\circ (\bar{h}_0\otimes \bar{f}_1\bar{\sigma}_{m+1})=
\overline{-D-b_2\circ (h_0\otimes f_1 g_{m+1})}
\]
in 
\begin{equation}
\label{eq:quasi4}
H^\ast(\Hom((\Sigma \bar{\bb})^{\otimes m+1},\Sigma \bb)=
\Hom(\Sigma H^\ast(\bar{\bb})^{\otimes m+1},\Sigma H^\ast(\bb))
\end{equation}
where $\overline{?}$ denotes cohomology classes or actions on cohomology.
Using the fact that $\bar{f}_1=H^\ast(f_1)$ is an equivalence and
$H^\ast(h_0)$ is a natural isomorphism one easily sees that 
\eqref{eq:quasi4} has a (unique) solution. 
\end{proof}
We will need the following variant of Lemma \ref{lem:quasi-inverse} which is proved in a similar way.
\begin{lemma} \label{lem:ffcase}
Let $\aa$, $\bb$ be $A_n$ categories for $n\ge 3$ and let $f:\aa\r\bb$ be a fully faithful $A_n$-functor
which is also a quasi-equivalence. Then there exists an $A_{n-1}$-quasi-equivalence $g:\bb\r \aa$ such
that $fg$
and $\Id_\bb$ are homotopic and such that $gf=\Id_\aa$.  Moreover the quasi-inverse $H^*(g_1)$ to $H^\ast(f_1)$ may be chosen freely.
\end{lemma}
\subsection{The category {\mathversion{bold} $\Free(\aa)$}.}
\label{sec:free}
\begin{definition}
\label{def:free}
Given an $A_n$-category $\aa$, $\Free(\aa)$ is obtained from $\aa$ by formally adding finite (possibly empty) direct
  sums and shifts of objects in $\aa$, i.e. an object of $\Free(\aa)$
  is given by 
\begin{equation}
\label{eq:free}
 A=\oplus_{i\in I} \ \Sigma^{a_i} A_i
 \end{equation}
where $A_i\in\Ob(\aa)$, $a_i\in \ZZ$, $|I|<\infty$. We allow $|I|=\emptyset$. 
Morphisms in $\aa$ are defined as
\[ 
\Free(\aa)(\oplus_i  \Sigma^{a_i} A_i,\oplus_j \Sigma^{b_i}B_j)=
\oplus_{i,j} \Sigma^{b_i-a_i}\aa(A_i,B_j).
\]
An element $f\in \aa(A,B)$ considered as an element of $\Free(\aa)(\Sigma^a A,\Sigma^b B)$ 
will be written as $\sigma^{b-a}f$ such that $|\sigma^{b-a}f|=|f|-(b-a)$.

 If $\aa$ is an $A_n$-category we can then make $\Free(\aa)$ into an
$A_n$-category. We need to define the higher compositions between morphisms between objects of
the form $\Sigma^a A$ (the case of more complicated objects is done by linear extension).
So if we have we have
maps in $\aa$:
\[
A_0\xrightarrow{f_1}A_1\xrightarrow{f_2}\cdots\xrightarrow{f_n} A_n
\]
and corresponding maps in $\Free(\aa)$
\[
\Sigma^{a_0} A_0\xrightarrow{\sigma^{a_1-a_0}f_1}\Sigma^{a_1}
A_1\xrightarrow{\sigma^{a_2-a_1}f_2}\cdots\xrightarrow{\sigma^{a_n-a_{n-1}}f_n} \Sigma^{b_n}A_n
\]
then
\[
b_n(s\sigma^{a_n-a_{n-1}}f_n,\ldots,s\sigma^{a_2-a_1}f_2,s\sigma^{a_1-a_0}f_1)
=\pm \sigma^{a_n-a_0}b_n(sf_n,\ldots,sf_2,sf_1)
\]
where the sign  is determined by  the usual Koszul sign convention (used with the rule $s\sigma=-\sigma s$).
\end{definition}
The $A_n$-category $\Free(\aa)$ is equipped with an strict $A_n$-endo functor $\Sigma$  such
that on objects we have
\begin{equation}
\label{eq:shiftfree}
\Sigma\left( \oplus_i  \Sigma^{a_i} A_i\right)=\oplus_i  \Sigma^{a_i+1} A_i
\end{equation}
and on morphisms $\Sigma$ is given by $\Sigma(\sigma^af)=(-1)^a\sigma^a f$ for $f$ a morphism in $\aa$. We will call $\Sigma$ the shift functor
on $\Free(\aa)$. Likewise $\Free(\aa)$ is equipped with an (associative) operation $\oplus$ with
an obvious definition. We will call it the ``direct sum''. Finally if $I=\emptyset$ in \eqref{eq:free}  the resulting object is denoted by $0$ and is called the ``zero object''.
\section{Truncated twisted complexes}
\label{sec:truncated}
From now on let $\aa$ be an $A_n$-category. 
\subsection{Higher cone categories}
Let $\aa^{\oplus m}$ be the graded graph whose objects are formal direct sums of precisely $m$ objects in $\aa$.
 \begin{equation}
\label{eq:objectA}
A=A_0\oplus A_1 \oplus \ldots \oplus A_{m-1},
\end{equation}
Morphisms are given by
\begin{equation}
\label{eq:linearHoms}
 \aa^{\oplus m}(A,B)= \oplus_{i,j=0}^{m-1} \aa(A_i,B_j).
\end{equation}
We extend the higher operations on $\aa$ linearly to $\aa^{\oplus m}$ so that $\aa^{\oplus m}$ becomes
an $A_n$-category.
\begin{remark}\label{rem:convention}
 Below we usually think of objects in $\aa^{\oplus m}$ as column vectors and similarly of morphisms in $\aa^{\oplus m}$ as
matrices acting on those column vectors.
\end{remark}
\begin{definition}[Higher cone categories]
Assume $m\le n+1$. The graded graph $\aa^{*m}$ is defined as follows.
\begin{itemize}
\item Objects are given by couples $(A,\delta_A)$ such that $A\in \Ob(\aa^{\oplus m})$
and  $\delta_A \in \aa^{\oplus m}(A,A)_1$
  is a 
``Maurer-Cartan element'' with strictly lower triangular matrix, i.e. it  satisfies
\begin{equation}
\label{eq:mcequation}
\sum_{i\le m-1} b_i(s\delta_A,\ldots,s\delta_A)=0
\end{equation}
\item Morphisms are given by
\begin{equation}
\label{eq:coneHoms}
 \aa^{*m}((A,\delta_A),(B,\delta_B))= \aa^{\oplus m}(A,B)
\end{equation}
\end{itemize}
\end{definition}
\begin{lemma}\label{higher operations} Assume $m\le n+1$.
The graded graph $\aa^{*m}$ has the structure of an $A_{ \left \lfloor{ \frac{n-m+1}{m}}\right \rfloor}$-category with higher multiplications given by
\begin{equation}\label{MC multiplication}
b_{\aa^{*m},i}(sg_i,\ldots,sg_1)=\sum_{\substack{l_0,\ldots,l_i \\ h=i+\Sigma l_j\le n}} b_{\aa^{\oplus m},h} (\underbrace{s\delta_i,\ldots,s\delta_i}_{l_i},sg_i,
\underbrace{s\delta_{i-1},\ldots,s\delta_{i-1}}_{l_{i-1}},\ldots,\underbrace{s\delta_1,\ldots,s\delta_1}_{l_1}
,sg_1,\underbrace{s\delta_0,\ldots,s\delta_0}_{l_0}) 
\end{equation}
for any set 
\[
(B_0,\delta_{0})\xrightarrow{g_1} (B_1,\delta_{1})\xrightarrow{g_2}\cdots\xrightarrow{g_i} (B_i,\delta_i).
\]
of $i\le \lfloor (n-m+1)/m\rfloor$ composable arrows  in $\aa^{*m}$.
\end{lemma}
\begin{proof}
We need to check $b_{\aa^{\ast m}}\circ b_{\aa^{\ast m}}=0$ on $i$ composable arrows for $i\le \lfloor (n-m+1)/m\rfloor$ as well as the correct behavior of identities. We will concentrate on the first condition
as it is the most interesting one.
As we will use similar facts several times below we present the argument in some detail. 

If we expand  $(b_{\aa^{\ast m}}\circ b_{\aa^{\ast m}})_i$ then it becomes
the sum of multilinear expressions evaluated on lists of arguments of the form
\begin{equation}
\label{eq:arguments}
\underbrace{s\delta_i,\ldots,s\delta_i}_{l_i},sg_i,
\underbrace{s\delta_{i-1},\ldots,s\delta_{i-1}}_{l_{i-1}},\ldots,\underbrace{s\delta_1,\ldots,s\delta_1}_{l_1}
,sg_1,\underbrace{s\delta_0,\ldots,s\delta_0}_{l_0}
\end{equation}
The crucial point is that those multilinear expression are obtained by linear expansion of the 
corresponding expressions evaluated on composable arrows in $\aa$.
Now  for each element $(A,\delta_A)\in \aa^{*m}$, the Maurer-Cartan 
  element $\delta_A$ is a strictly lower triangular $m\times m$-matrix and 
  hence such extended expressions are zero on \eqref{eq:arguments} whenever one of the $l_j$ is $\ge m$.

By the assumption
\[ 
i\le \left \lfloor{ \frac{n-m+1}{m}}\right \rfloor
\] 
we obtain that the length of the relevant lists of arguments in \eqref{eq:arguments} is 
\begin{align*}
&\le (m-1)(i+1)+i\\
&=mi+m-1\\
&\le m\left\lfloor \frac{n-m+1}{m}\right \rfloor+m-1\\
&\le n-m+1+m-1\\
&= n
\end{align*}
Now the condition $b_{\aa^{\ast m}}\circ b_{\aa^{\ast m}}=$ combined
with \eqref{eq:mcequation} becomes $b_{\aa^{\oplus m}}\circ b_{\aa^{\oplus m}}=0$
when evaluated on lists of $\le n$ arguments.  This 
holds since $\aa^{\oplus m}$ is an $A_n$-category.
\end{proof}
Below we call $\aa^{\ast m}$ a higher cone category. This is motivated by Definition \ref{def:cone} below. 
\begin{lemma}[Functoriality of $*$]\label{lem:functoriality of *} Given $A_n$-categories $\aa$ and $\bb$ and $t\leq m+1 \leq n+1$, we obtain a strict $A_p$-functor
\[
*^t: A_m(\aa, \bb) \to A_{ \left \lfloor{ \frac{m-t+1}{t}}\right \rfloor}(\aa^{*t}, \bb^{*t})
\]
for $p=  \lfloor (n-t+1)/t \rfloor- \lfloor (m-t+1)/t \rfloor$. Moreover $\ast^t$ is strictly compatible with the compositions
\[
A_m(\bb,\cc)\times A_m(\aa,\bb)\r A_m(\aa,\cc).
\]
\end{lemma}

\begin{proof}
Since we are defining a strict functor we only need to define $(\ast^t)_1$.
We will write $(-)^{\ast^t}$ for $(\ast^t)_1(-)$.

First of all we define the functor on ``objects''. For an element  $f\in \Ob(A_m(\aa,
\bb))$ and $(A,\delta_A)\in \Ob(\aa^{\ast t})$ define
\[
f^{*t}(A,\delta_A)=(f(A),\sum_{i\leq t-1}f(s\delta_A,\ldots,s\delta_A))
\]
where $f$ is understood to be extended linearly to direct sums.
For a sequence of
composable arrows
\begin{equation}
\label{eq:composable}
(A_0,\delta_0)\xrightarrow{a_1} (A_1,\delta_1) \xrightarrow{a_2} \cdots \xrightarrow{a_d} (A_d,\delta_d)
\end{equation}
put
\[
(f^{*t})_d(sa_d,\ldots,sa_1)=\sum f_{d+i_0+\cdots+i_d}(s\delta_d^{\otimes i_d},sa_{d},s\delta_{d}^{\otimes i_{d-1}},\ldots,sa_1,s\delta_0^{\otimes i_0}).
\]
To show that $*^t$ sends an $A_m$-functor to an $A_{\lfloor (m-t+1)/t \rfloor}$-functor (i.e an element of 
$\Ob(A_{ \lfloor (m-t+1)/t \rfloor}(\aa^{*t}, \bb^{*t})$), one proceeds in the same way as in the proof of Lemma \ref{higher operations}.  

Now we define $(*^t)_1$ on $\Hom$-spaces in $A_m(\aa,\bb)$. Given
$f,g\in A_m(\aa,\bb)$ and $h\in A_m(\aa,\bb)(f,g)$ we define
$h^{*t}\in \Hom(f^{*t},g^{*t})$ as follows: for a sequence of
composable arrows
as in \eqref{eq:composable}
we have
\[
(h^{*t})_d(sa_d,\ldots,sa_1)=\sum h_{d+i_0+\cdots+i_d}(s\delta_d^{\otimes i_d},sa_{d},s\delta_{d-1}^{\otimes i_{d-1}},\ldots,sa_1,s\delta_0^{\otimes i_0}).
\]
One verifies that $(\ast^t)_1$ commutes with the higher operations on $A_m(\aa, \bb)$ and  $A_{\lfloor (m-t+1)/t \rfloor}(\aa^{*t}, \bb^{*t})$ (see Lemma \ref{lem:An-m}) and hence
defines a strict functor. It is an $A_p$-functor since $A_{ \left \lfloor{ \frac{m-t+1}{t}}\right \rfloor}(\aa^{*t}, \bb^{*t})$ is an $A_p$-category by Lemma \ref{higher operations} and Lemma \ref{lem:An-m}.
The strict compatibility with compositions is also a standard verification.
\end{proof}
\subsection{Truncated twisted complexes}
In the
$A_n$-category setting, untruncated twisted complexes are not well 
behaved as they form only a graded graph. Indeed even the
definition of the differential on morphisms between twisted complexes involves higher operations
of unbounded arity.
Therefore in this section we introduce  \emph{truncated twisted complexes} over an $A_n$-category. 
In this case the resulting object is still an $A_p$-category for some $p$, although
$p$ is much smaller than $n$.
\begin{definition}[Truncated twisted complexes] 
\label{def:truncated} Assume $m\le n$.
We define the truncated twisted complexes over $\aa$ as 
\[ 
\Tw_{\leq m}\aa = \Free(\aa)^{*m+1}  
\]
The map
\[
(A,\delta_A)\mapsto (A\oplus 0,(\delta_A,0))
\]
defines a fully faithful functor $\Tw_{\le m}\aa\r \Tw_{\le m+1}\aa$ which we will treat as an inclusion.
With this convention we write $\Tw\aa$ for $\bigcup_m \Tw_{\le m} \aa$ in case $\aa$ is an $A_\infty$-category. In a similar
vein we define the fully faithful functor $\Phi:\aa\r\Tw_{\le m}\aa:A\mapsto (A\oplus 0\oplus\cdots,0)$ which 
again we will treat as an inclusion.
\end{definition}
From Lemma \ref{higher operations} we obtain
\begin{lemma}\label{An-ness} Assume $m\le n$.
The category of truncated twisted complexes $\Tw_{\leq m}\aa $ has a structure of an $A_{ \left \lfloor{ \frac{n-m}{m+1}}\right \rfloor}$-category.
\end{lemma}
\begin{lemma}[Functoriality of $\Tw$]\label{functoriality of Tw}
Let $F:\aa \to \bb$ be an $A_m$-functor between two $A_n$-categories with $a\le m\leq n$. Then we obtain a corresponding $A_{ \left \lfloor{ \frac{m-a}{a+1}}\right \rfloor}$ functor
\[ 
\Tw_{\leq a} F:  \Tw_{\leq a}\aa \to  \Tw_{\leq a}\bb.
\]
Moreover $\Tw_{\leq a}(-)$ is strictly compatible with compositions of $A_n$-functors.
\end{lemma}

\begin{proof}
This follows immediately from Lemma \ref{lem:functoriality of *}.
\end{proof}
\subsection{Distinguished triangles}
\begin{definition} \label{def:cone} Assume $f:A\r B$ is closed morphism in $\aa$ of degree zero. Then
$C(f)$ is the object  $(\Sigma A\oplus  B,\delta_{C(f)})\in \Tw_{\leq 1}\aa$ such that 
\[
\delta_{C(f)}=
\begin{pmatrix}
0&0\\
\sigma^{-1}f&0
\end{pmatrix}
\]
(recall that we write objects as column vectors and morphisms as matrices - see Remark \ref{rem:convention}).% \marginpar{Perhaps we should specify a lower bound for $n$ here. At least we need $\geq 1$ to talk about closed morphisms. }
\end{definition}
\begin{definition}
Let $f:A\r B$ be a morphism in $Z^0\aa$. The associated \emph{standard distinguished triangle} $\delta_f$ in $\Tw_{\leq 1}\aa$ is given by
\begin{equation} 
\label{eq:standard}
A\xrightarrow{f} B \xrightarrow{i} (C(f),\delta_{C(f)}) \xrightarrow[(1)]{p} A 
\end{equation}
where
\[
i=\begin{pmatrix}
0\\
\Id_B
\end{pmatrix}\qquad
p=\begin{pmatrix}
\sigma^{-1}\Id_{A}&0
\end{pmatrix}
\]
The image of $\delta_f$ in $H^0(\Tw_{\le 1}\aa)$ is written as $\bar{\delta}_f$. It is also
called a standard distinguished triangle.
\end{definition}
\begin{definition}\label{triangulatedstructure}
  Let $\aa$ be an $A_n$-category with $n\ge  7$. A triangle in $H^\ast(\aa)$
  is said to be \emph{distinguished} if its image under $H^\ast(\Phi)$ is
  isomorphic to a standard distinguished triangle in
  $H^\ast(\Tw_{\leq 1}\aa)$. 
\end{definition}
From this definition we immediately obtain:
\begin{theorem} \label{bla}
Let $\rho:\aa\r\bb$ be an $A_m$-functor between $A_n$-categories for $m\ge 5$, $n\ge 7$. Then $H^\ast(\rho)$ preserves distinguished triangles.
\end{theorem}
\begin{proof}
It is clear that there is a commutative diagram
\[
\xymatrix{
\aa\ar[r]^-{\Phi}\ar[d]_{\rho}& \Tw_{\le 1}\aa\ar[d]^{\Tw_{\le 1}\rho}\\
\bb\ar[r]_-{\Phi}&\Tw_{\le 1}\bb
}
\]
By Lemma \ref{An-ness} $\Tw_{\le 1}\aa$ and  $\Tw_{\le 1}\bb$ are $A_3$-categories and by Lemma \ref{functoriality of Tw}, $\Tw_{\le 1}\rho$ is an $A_2$-functor. Hence $H^\ast(\Tw_{\le 1}\rho)$ is a graded functor (see \S\ref{sec:Ancatsandfunctors}).
One checks $H^\ast(\Tw_{\le 1}\rho)(\bar{\delta}_f)=\bar{\delta}_{\rho_1(f)}$. This implies what we want.
\end{proof}
\section{DG-categories}
\subsection{Generalities}
Recall that a DG-category is an $A_\infty$-category such that $m_i=0$ for $i\ge 3$. It that case $\Tw \aa$ is also a DG-category.
We recall the following definition.
\begin{definition} \cite{Bondal5}
A DG-category is \emph{pre-triangulated} if the DG-functor
$\Phi:\aa \r \Tw\aa$ is a quasi-equivalence.
\end{definition}
The main result concerning pre-triangulated DG-categories is 
\begin{theorem} \cite{Bondal5} \label{th:bondal}
If $\aa$ is pre-triangulated then $H^\ast(\aa)$, when equipped with distinguished triangles as in Definition \ref{triangulatedstructure}, is triangulated.
\end{theorem}
\begin{proof} Assume first that $\aa$ is a general DG-category. 
$\Tw\aa$ is equipped with a natural cone functor $C(f)$ and
a notion of standard triangles $\delta_f$ for any closed map $f:A\r B$:
\[
A\xrightarrow{f} B\xrightarrow{i} C(f)\xrightarrow[(1)]{p} A.
\]
A triangle in $\Tw\aa$ is called distinguished if it is isomorphic
to a standard triangle. 
In \cite{Bondal5} it is proved  that 
$H^\ast(\Tw \aa)$ is triangulated when equipped with this class of distinguished triangles. If $\aa$ is pre-triangulated then $H^\ast(\aa)$ inherits the triangulated
structure from $H^\ast(\Tw \aa)$. We have to prove that the distinguished
triangles are the same as those in Definition \ref{triangulatedstructure}.
Assume that 
\[
\bar{\delta}: A\xrightarrow{f} B\r C\xrightarrow[(1)]{} A
\]
is a triangle in $H^\ast(\aa)$ distinghuished in the sense of \cite{Bondal5}, i.e. $\Phi(\bar{\delta})$ is distinguished in $H^\ast(\Tw \aa)$. Now
$\bar{\delta}_{\Phi(f)}\in H^\ast(\Tw_{\le 1}\aa)$ is a distinguished triangle in $H^\ast(\Tw\aa)$ which
has the same base as $\bar{\delta}_f$. By the axioms for triangulated
categories we conclude that $\Phi(\bar{\delta})\cong \bar{\delta}_{\Phi(f)}$.
Hence $\bar{\delta}$ is distinguished in the sense of Definition  \ref{triangulatedstructure}. The opposite direction is similar.
\end{proof}
\subsection{Some small $DG$-categories}
\label{sec:diagcats}
\begin{definition} Let $n\ge 0$. Then $I_n$ is the DG-category with objects $(x_i)_{i=0}^n$ such that
\[
I_n(x_i,x_j)=
\begin{cases}
ka_{ij}&\text{if $i<j$}\\
k\Id_{x_i}&\text{if $i=j$}\\
0&\text{otherwise}
\end{cases}
\]
with $|a_{ij}|=0$, $a_{jl}a_{ij}=a_{il}$ and $da_{ij}=0$. We will write $a_i=a_{i,i+1}$ for $i=0,\ldots n-1$.
\end{definition}
\begin{lemma} 
\label{lem:small-pretri} $\Tw_{\le 1} I_n$ is pre-triangulated.
\end{lemma}
\begin{proof} As $\Tw I_n$ is pre-triangulated \cite{Bondal5} it is sufficient
to prove that $H^\ast(\Tw_{\le 1} I_n)\r H^\ast(\Tw I_n)$ is essentially
surjective. This is essentially \cite[Proposition 7.27]{Rouquier}.
For the convenience of the reader we repeat the argument.
The Yoneda embedding  realizes $H^\ast(\Tw I_n)$ as the bounded
derived category $D^b(\rep(I_n))$ of the representations of $I_n$, viewed as quiver. Since
$\rep(I_n)$ is a hereditary category every object in $D^b(\rep(I_n))$ 
is the direct sum of its (shifted) cohomology objects which are in $\rep(I_n)$. Moreover every object in $\rep(I_n)$ has
projective dimension one and so it is isomorphic to a single cone
of objects in $\Free(I_n)$. In other words it is in the essential image
of $H^\ast(\Tw_{\le 1} I_n)$.
\end{proof}
\begin{remark} \label{rem:small-pretri} Assume $n=0$. Then $\rep(I_0)$ has global dimension zero and
we have in fact that $\Free(I_0)=\Tw_{\le 0}I_0$ is pre-triangulated.
\end{remark}
\section{Pre-triangulated $A_n$-categories}
From now on let $\aa$ be an $A_n$-category. The purpose of this section is to define what it means for $\aa$ to be pre-triangulated and to show that this definition implies that $H^\ast(\aa)$ is triangulated. 
\begin{definition} 
An $A_n$-category $\aa$, with $n\geq 7$, is said to be \emph{pre-triangulated} if the inclusion
$ \aa \xrightarrow{\Phi} \Tw_{\leq 1} \aa $
is a quasi-equivalence.
\end{definition}
\begin{remark} The lower bound $n\geq 7$ comes from the fact that we want 
 $H^\ast(\Tw_{\le 1}\aa)$ to be an honest category. This happens when $\Tw_{\le 1}\aa$ is an $A_3$-category.
For this to be true $\aa$ needs to be at least an $A_7$-category by Lemma \ref{An-ness}. 
\end{remark}
\begin{theorem}
\label{th:mainth}
Let $\aa$ be a pre-triangulated $A_n$-category for $n\ge 13$.
When equipped with the collection of distinguished triangles as in Definition \ref{triangulatedstructure},  $H^\ast(\aa)$ is a triangulated category.
\end{theorem}
\begin{proof}
Here is the ``strategy'': we have to prove that $H^\ast(\aa)$ satisfies TR0-TR4 as in \S\ref{sec:trianggraded}. For the TR1-TR4 axioms we will translate their input into a suitable $A_n$-functor $\mu:I_m\r \aa$, for $m\le 2$, which is then
extended to an $A_{\lfloor (n-1)/2\rfloor}$-functor $\Tw_{\le 1}\mu:\Tw_{\le 1}I_m\r \Tw_{\le 1} \aa$.
Then we use that $\Tw_{\le 1}I_m$ is pre-triangulated by Lemma \ref{lem:small-pretri} and hence in particular $H^\ast(\Tw_{\le 1}I_m)$ is  triangulated by Theorem \ref{th:bondal}. 
Roughly speaking we then transfer the output of the TR1-TR4-axioms for $H^\ast(\Tw_{\le 1}I_m)$ to $H^\ast(\aa)$ by using Theorem \ref{bla}.

To accomplish the last step we will pick an $A_p$-functor $\pi: \Tw_{\leq 1} \aa\to \aa$, for 
$p=\lfloor (n-1)/2\rfloor-1=\lfloor (n-3)/2\rfloor$, which is a homotopy inverse to $\Phi$ such that $\pi\Phi$ is the identity (see Lemmas \ref{An-ness}, \ref{lem:ffcase}).
In particular we have that $H^\ast(\Phi)$ and $H^\ast(\pi)$ are quasi-inverses to each other. Since $n\ge 13$,
$\Tw_{\le 1}\mu$ is at least an $A_6$-functor and $\pi$ is at least an $A_5$-functor. 
So $H^\ast(\pi \Tw_{\le 1}\mu)$ preserves distinguished triangles by Theorem \ref{bla}. % \marginpar{This might actually be enough?}
To avoid making some arguments
needlessly cumbersome we will in fact also use
that $H^\ast(\Tw_{\le 1}\mu)$ preserves standard distinguished triangles and that $H^\ast(\pi)$ sends a standard distinguished triangle in $H^\ast(\Tw_{\le 1}\aa)$ to a distinguished triangle in $H^\ast(\aa)$. The latter
follows easily from the fact that $H^\ast(\pi)$ is a quasi-inverse to $H^\ast(\Phi)$. Note that the intermediate category $\Tw_{\le 1}\aa$ may be only an $A_6$-category
so, with our current definitions, we cannot talk about distinguished triangles in\footnote{We could have eliminated this minor technical complication by simply requiring $n\ge 15$.}  $H^\ast(\Tw_{\le 1}\aa)$.
\begin{enumerate}
\item[TR0]  
Like $\Free(\aa)$ (see \S\ref{sec:free}), $\Tw_{\le 1}\aa$ is equipped with canonical operations $\Sigma$ and $\oplus$. These descend to operations 
on $H^\ast(\Tw_{\le 1}\aa)$ which one easily checks to be to be the categorical direct sum and shift functor.
Since $H^\ast(\aa)\r H^\ast(\Tw_{\le 1}\aa)$ is an equivalence, the direct sum and shift functor
defined on $H^\ast(\Tw_{\le 1}\aa)$ descend to $H^\ast(\aa)$.
\item[TR1] First we note that the triangle 
\begin{equation}
\label{eq:trivial}
A\xrightarrow{\id_A}A\xrightarrow{} 0\xrightarrow[(1)]{} A
\end{equation}
 is distinguished. Indeed: the functor $\mu:I_0\r \aa:x_0\mapsto A$ extends to a functor\footnote{The reader will note that here the literal execution of our ``strategy'' is a bit uneconomical and that by Remark \ref{rem:small-pretri} we could have used
$\Tw_{\le 0}I_0$.}
 $\mu:\Tw_{\le 1}I_0\r \Tw_{\le 1}\aa$ and \eqref{eq:trivial} is the image under $H^\ast(\pi\Tw_{\le 1}\mu)$ of the distinguished triangle in $H^\ast(\Tw_{\le 1}I_0)$ (which satisfies TR1)
\[
x\xrightarrow{\id_{x_0}}x\xrightarrow{0} 0\xrightarrow[(1)]{0} x
\]
Now we prove the second part of the TR1 conditions: the existence of distinguished triangles with a given base. Consider a map $A\xrightarrow{f}B$ in $H^\ast(\aa)$ and put $\bar{\delta}=H^\ast(\pi)(\bar{\delta}_f)$. Since $\bar{\delta}_f$ is a standard distinguished triangle in $\Tw_{\leq 1} \aa$, $\bar{\delta}$ is distinguished.
%As $H^\ast(\Phi)(\delta)\cong \bar{\delta}_f$, $\delta$ is distinguished.
%\marginpar{Alternatively, one can take the strict $A_n$-functor $\mu:I_1\r \aa$ which sends $a_0$ to $f$. Then use that $H^\ast(\Tw_{\le 1}I_1)$ satisfies TR1 so that we can complete that morphism to a distinguished triangle $\bar{\delta}_{a_0}$, and apply  $H^\ast(\pi\Tw_{\le 1}\mu)$ }
\medskip

Finally, the fact that any triangle isomorphic to a distinguished triangle is distinguished follows immediately from Definition \ref{triangulatedstructure}.
\item[TR2]
Let $\bar{\delta}$ be a distinguished triangle in $H^\ast(\aa)$. Then there exists an isomorphim with a standard triangle $H^\ast(\Phi)(\bar{\delta})\cong \bar{\delta}_f$ and
hence in particular $\bar{\delta}\cong H^\ast(\pi)(\bar{\delta}_f):=\bar{\delta'}$. 
There is a strict $A_n$-functor $\mu:I_1\r \aa$ which sends $a_0$ to $f$ and $\bar{\delta}_f$ is the image of $\bar{\delta}_{a_0}\in H^\ast(\Tw_{\le 1}I_1)$
under the morphism $H^\ast(\Tw_{\le 1}\mu)$.
Since $H^\ast(\Tw_{\le 1}I_1)$ satisfies TR2, the rotated versions of $\bar{\delta}_{a_0}$ are distinguished in $H^\ast(\Tw_{\le 1} I_1)$
and we obtain rotated versions of $\bar{\delta}'$ by applying $H^\ast(\pi\Tw_{\le 1}\mu)$ (note that a graded functor preserves suspensions and desuspensions).
By TR1 the corresponding rotated versions of $\bar{\delta}$ are
also distinguished.
\item[TR3] 
Suppose we have a diagram of distinguished triangles in $H^0(\aa)$
\begin{equation}\label{TR3diagram}
\xymatrix{
A \ar[r]^f \ar[d]_u&B\ar[r]\ar[d]^v &C \ar[r]^{(1)} &A\ar[d]^u\\
A' \ar[r]_{f'} & B' \ar[r] & C' \ar[r]_{(1)} & A'
}\end{equation}
Up to composing with an isomorphism of triangles, we can assume that the two distinguished triangles in the diagram are standard distinguished triangles in $\Tw_{\le 1}\aa$ so that
$C=C(f)$, $C'=C(f')$. Hence we have to construct the dotted arrow
in
\[
\xymatrix{
A \ar[r]^f \ar[d]_u&B\ar[r]\ar[d]^v &C(f)\ar@{.>}[d]^w \ar[r]^{(1)} &A\ar[d]^u\\
A' \ar[r]_{f'} & B' \ar[r] & C(f') \ar[r]_{(1)} & A'
}
\]
It is easy to give a formula for $w$. Alternatively one may lift the
square on the left to an $A_n$-functor $I\otimes I\r \aa$ and 
then proceed by considering the induced functor $\Tw_{\le 1}(I\otimes I)\r \Tw_{\le 1}\aa$.

We will give instead a proof compatible with our ``strategy''.  By writing the solid square as a composition of 2 squares
it is sufficient to consider the case in which either $u$ or $v$ is the identity. The two cases are similar so we will consider the first one. Now
the diagram is
\begin{equation}
\label{eq:basic}
\xymatrix{
A \ar[r]^f \ar@{=}[d]&B\ar[r]\ar[d]^v &C(f)\ar@{.>}[d]^w \ar[r]^{(1)} &A\ar@{=}[d]\\
A \ar[r]_{f'u} & B' \ar[r] & C(f'u) \ar[r]_{(1)} & A
}
\end{equation}
We may construct an $A_n$-morphism $\mu:I_2\r \aa$ such that $\mu_1(a_0)=f$, $\mu_1(a_1)=v$, $\mu_1(a_1a_0)=f'u$ (note that we need a non-trivial $\mu_2$ as $vf$ is not necessarily equal to $f'u$ in $\aa$). Inside $H^\ast(\Tw_{\le 1}I_2)$ we have
the diagram
\[
\xymatrix{
x_0 \ar[r]^{a_0} \ar@{=}[d]&x_1\ar[r]\ar[d]^{a_1} &C(a_0)\ar@{.>}[d] \ar[r]^{(1)} 
&x_0\ar@{=}[d]\\
x_0 \ar[r]_{a_1a_0} & x_2 \ar[r] & C(a_1a_0) \ar[r]_{(1)} & x_0
}
\]
where now the dotted arrow exists as $H^\ast(\Tw_{\le 1}I_2)$ satisfies TR3.
Applying $H^\ast(\Tw_{\le 1}\mu)$ we obtain \eqref{eq:basic}.
\item[TR4]
Since we have shown TR1-TR3, by \cite[1.1.6]{BBD} it suffices to show that
any  composable pair of degree zero morphisms $X\r Y\r Z$ in  $H^\ast(\aa)$
can be completed to an octahedron as in \eqref{eq:octahedron1}.

A composable pair of degree zero morphisms in $H^\ast(\aa)$ can be lifted to an $A_n$-functor
 $\mu:I_2\r \aa$.
The image of the octahedron in  $H^\ast(\Tw_{\le 1}I_2)$
built on 
$x_0\xrightarrow{a_0}x_1\xrightarrow{a_1}x_2$
under $H^\ast(\pi \Tw_{\le 1}\mu)$ is now the sought octahedron in~$H^\ast(\aa)$. \qedhere
\end{enumerate}
\end{proof}
\section{Gluing $A_n$-categories}
\subsection{Bimodules}
\label{sec:bimodules}
Let $\aa$, $\bb$ be $A_n$-categories. An $A_{n+1}$-$\bb$-$\aa$-bimodule is
collection of graded vector spaces $M(A,B)$, $A\in \Ob(\aa)$,
$B\in \Ob(\bb)$ together with a codifferential on
$(B^+ \bb\otimes M\otimes B^+\aa)_{\le n+1}$ where the latter is regarded as a
DG-$(B^+\aa)_{\le n}{-}(B^+\bb)_{\le n}$-bicomodule.  In other words,
such a bimodule is equipped with higher operations of degree one
\begin{equation}
\label{eq:higher}
b_M: \Sigma\bb(B_{p-1},B_{p})\otimes \cdots \otimes \Sigma\bb(B_{a+1},B_{a+2}) \otimes M(A_a,B_{a+1})\otimes \Sigma\bb(A_{a-1},A_{a})\otimes \cdots \otimes \Sigma\bb(A_0,A_1)
\r M(A_0,B_p)
\end{equation}
for $(A_i)_i\in \Ob(\aa)$, $(B_j)_j\in \Ob(\bb)$
for $p\le n+1$ such that $b\circ b=0$. In addition we require that the higher operations vanish on identities, when appropriate.
If $\aa=\bb$ then the \emph{identity $A_n$-$\aa$-bimodule} is given by $M(A,A')=\aa(A,A')$ and the higher operations are those of $\aa$.

\medskip

If $\aa_1,\aa_2,\bb_1,\bb_2$ are $A_n$-categories, $f_i:\aa_i\r\bb_i$ are $A_{n}$-functors and $M$ is an $A_{n+1}$- $\bb_2$-$\bb_1$-bimodule
  then we write ${}_{f_1} M_{f_2}$ for the $\aa_2$-$\aa_1$-bimodule which is the pullback of $M$
along $(f_1,f_2)$. For $A_1\in \Ob(\aa_1)$, $A_2\in \Ob(\aa_2)$ we have ${}_{f_1}M_{f_2}(A_1,A_2)=M(f_1(A_1),f_2(A_2))$
and the higher operations on ${}_{f_1} M_{f_2}$
are schematically given by the following formula for $m\in {}_{f_1}M_{f_2}(A_1,A_2)$
\[
b_{{}_{f_1}M_{f_2}}(\ldots,m,\ldots)=\sum\pm b_M(f_2(\ldots),\ldots,f_2(\ldots),m,f_1(\ldots),\ldots,f_1(\dots))
\]
(the sign is given by the Koszul convention). It is easy to see that ${}_{f_1} M_{f_2}$ is an $A_{n+1}$-bimodule. If
$f_1$ or $f_2$ is the identity then we omit it from the notation.

\subsection{The arrow category}
\begin{definition}[The arrow category]
  Let $\aa$, $\bb$ be $A_n$-categories and let $M$ be a
  $\bb$-$\aa$-$A_n$-bimodule. 
The arrow category
  $\cc=\aa\xrightarrow{M}\bb$ has $\Ob(\cc)=\Ob(\bb)\coprod \Ob(\aa)$
  and morphisms for $B,B'\in \Ob(\bb)$, $A,A'\in \Ob(\bb)$ given by
  $\cc(A,A')=\aa(A,A')$, $\cc(B,B')=\bb(B,B')$ and $\cc(A,B)=M(A,B)$, 
  $\cc(B,A)=0$.
\end{definition} 
  It is easy to see that $\aa\xrightarrow{M}\bb$ becomes an $A_n$-category by combining the higher multiplications on   
$\aa$, $\bb$ and $M$ (as in \eqref{eq:higher}).

\medskip

Assume we have $A_n$-categories $\aa$, $\bb$, $\aa'$, $\bb'$ and $\bb$-$\aa$ and $\bb'$-$\aa'$ bimodules $M$ and $M'$. Below it will be convenient to consider
the category $A^\circ_{m}(\aa\xrightarrow{M}\bb,\aa'\xrightarrow{M'}\bb')$ of $A_m$-functors $F:(\aa\xrightarrow{M}\bb)\r (\aa'\xrightarrow{M'}\bb')$ such
that $F(\Ob(\aa))\subset \Ob(\aa')$, $F(\Ob(\bb))\subset \Ob(\bb')$. It is easy to see that $F$ contains the same data as $A_m$-functors $F_\aa:\aa\r \aa'$, $F_\bb:\bb\r \bb'$ together
with a $A_m$-bimodule morphism $F_M:M\r {}_{F_\aa}M'_{F_\bb}$. Sometimes we will write $F=(F_\aa,F_M, F_\bb)$.

\subsection{The gluing category}
\begin{definition}[The gluing category]
Assume $n\ge 1$.  Let $\aa$, $\bb$ be $A_n$-categories and let $M$ be a
  $\bb$-$\aa$-$A_n$-bimodule. The gluing category $\aa\coprod_M \bb$ is the full graded subgraph of $(\aa\xrightarrow{M}\bb)^{\ast 2}$ given by
objects of the form $(A\oplus B,\delta)$ with $A\in \Ob(\aa)$ and $B\in\Ob(\bb)$ (note that $\delta$ is simply an element of $Z^1 M(A,B)$).
\end{definition}
\begin{lemma} \label{lem:Aness-gluing}
$\aa\coprod_M\bb$ has the structure of an $A_{n-1}$ category with higher multiplications given by \eqref{MC multiplication}.
\end{lemma}
\begin{proof} The proof is as in Lemma \ref{higher operations} except that now in the relevant argument lists in \eqref{eq:arguments} we can have at most one $\delta$, as the $(g_j)_j$ are now
represented by lower triangular $2\times 2$-matrices.
\end{proof}
\begin{remark} An alternative way of defining $\aa\coprod_M\bb$ is as follows. Let
$J_1$ be defined like $I_1$ (see \S\ref{sec:diagcats}) except that we put $|a_0|=1$. Then $\aa\coprod_M\bb$ may be identified with the full subcategory of $A_1(J_1,\aa\xrightarrow{M}\bb)$
consisting of $A_1$-functors $F:J_1\r (\aa\xrightarrow{M}\bb)$ such that $F(x_0)\in \Ob(\aa)$, $F(x_1)\in \Ob(\bb)$. It then follows from Lemma \ref{lem:An-m} that $\aa\coprod_M\bb$ is indeed
an $A_{n-1}$-category.
\end{remark}
The following will be our main result in this section.
\begin{theorem} \label{th:maingluing}
Assume that $n\ge 13$, that $\aa$, $\bb$ are pre-triangulated $A_n$-categories and that $M$ is an $A_n$-$\bb$-$\aa$-bimodule. Then $\aa\coprod_M\bb$ is a pre-triangulated
$A_{n-1}$ category. Moreover the obvious fully faithful functors $\varphi_\aa:H^\ast(\aa)\r H^\ast(\aa\coprod_M \bb)$, $\varphi_\bb:H^\ast(\bb)\r H^\ast(\aa\coprod_M \bb)$ preserve distinguished triangles. If $n\ge 14$ so that
$H^\ast(\aa\coprod_M\bb)$ is triangulated by Theorem \ref{th:mainth} and Lemma \ref{lem:Aness-gluing}
then $\varphi_\aa$, $\varphi_\bb$ give rise to a semi-orthogonal decomposition
\begin{equation}
\label{eq:semi-orth}
\textstyle H^\ast(\aa\coprod\nolimits_M\bb)=\langle H^\ast(\aa),H^\ast(\bb)\rangle
\end{equation}
whose associated bimodule (see \S\ref{intro:gluing}) is $H^\ast(M)$.
\end{theorem}
The proof of this theorem requires some preparation. We start with:
\begin{proposition}[Functoriality of gluing] \label{prop:gluing} Assume we have $A_n$-categories
  $\aa$, $\bb$, $\aa'$, $\bb'$ and $\bb$-$\aa$ and $\bb'-\aa'$
  bimodules $M$ and $M'$. Then for $m\le n$, there is a strict $A_{n-m}$-functor
\[
\textstyle \phi:A_m^\circ(\aa\xrightarrow{M}\bb,\aa'\xrightarrow{M'}\bb')\r A_{m-1}(\aa\coprod_M\bb,\aa'\coprod_{M'}\bb').
\]
Moreover $\phi$ is strictly compatible with compositions.
\end{proposition}
\begin{proof} This is proved like Lemma \ref{lem:functoriality of *} which also gives the relevant formulas (where we take into account that in this case at most one $\delta$
can appear in the relevant arguments lists in \eqref{eq:arguments}).
\end{proof}
\begin{corollary} \label{cor:quasi}
Let $3\le m\le n-3$ and let $\aa$, $\bb$, $\aa'$, $\bb'$, $M$, $M'$ be as in Proposition \ref{prop:gluing} and let $F\in A_m^\circ(\aa\xrightarrow{M}\bb,\aa'\xrightarrow{M'}\bb')$.  
If $F$ is a quasi-equivalence then so is $\phi(F)$.
\end{corollary}
\begin{proof}
Note that $F$ is a quasi-equivalence if and only if $F_\aa$, $F_\bb$ are quasi-equivalences and $F_M$ is a quasi-isomorphism.
By Lemma \ref{lem:quasi-inverse} we may choose an inverse $G\in A_{m-1}^\circ(\aa'\xrightarrow{M'}\bb',\aa\xrightarrow{M}\bb)$ to $F$, up to homotopy (making use of the fact that the quasi-inverse to $H^\ast(F_1)$ may be
chosen freely). Note that $H^\ast(G)$ is a functor as $m-1\ge 2$.

Since $H^\ast(\phi)$ also being a functor (as $n-m\ge 3$) preserves invertible maps, we conclude by Lemma \ref{lem:invertible}
that it preserves homototopies. Hence $\phi(G)$ is an inverse to
$\phi(F)$ up to homotopy. It follows that $H^\ast(\phi(F))$ is an
equivalence $H^\ast(\aa\coprod_M\bb)\r H^\ast(\aa'\coprod_{M'} \bb')$.
\end{proof}
For the next few results we assume that $\aa$, $\bb$ are $A_n$ categories and that $M$ is an $A_n$-$\bb$-$\aa$-bimodule.
We define $M^{\ast 2}$ as the $\bb{\ast} \bb$-$\aa{\ast} \aa$
bimodule such that $M^{\ast 2}((A_0\oplus A_1,\delta_A),(B_0\oplus B_1,\delta_B))=M(A_0,B_0)\oplus M(A_0,B_1)\oplus M(A_1,B_0)\oplus M(A_1,B_1)$ where the higher operations
on $M^{\ast 2}$ are obtained from those of $M$  by ``inserting Maurer-Cartan elements'' like in Lemma \ref{higher operations}.  
In a similar way as Lemma \ref{lem:functoriality of *} one proves
\begin{lemma} \label{lem:Aness:star:bimod}
$M^{\ast 2}$ is a $A_{\lfloor (n-1)/2\rfloor}$-bimodule.
\end{lemma}
\begin{lemma}
\label{lem:star}
Let $n\ge 3$ and let $\aa$, $\bb$, $M$ be as above. We have a fully faithful functor of  $A_{\lfloor (n-1)/2\rfloor-1}$-categories 
\begin{equation}
\label{eq:centraliso}
\textstyle (\aa\coprod_M\bb){\ast}(\aa\coprod_M\bb)\r \aa{\ast}\aa\coprod_{M^{\ast 2}} \bb{\ast}\bb
\end{equation}
\end{lemma}
\begin{proof} An object in $(\aa\coprod_M\bb){\ast}(\aa\coprod_M\bb)$ is of the form
\[
((A_0\oplus B_0,\delta_0)\oplus (A_1\oplus B_1,\delta_1),\delta)
\]
where  $\delta=(\delta_{00},\delta_{10},\delta_{11})\in \aa(A_0,A_1)_1\oplus M(A_0,B_1)_1\oplus \bb(B_0,B_1)_1$ is such that
\[
\delta=
\begin{pmatrix}
0&0&0&0\\
0&0&0&0\\
\delta_{00}&0&0&0\\
\delta_{01}&\delta_{11}&0&0
\end{pmatrix}
\text{ acting on }
\begin{pmatrix}
A_0\\B_0\\A_1\\B_1
\end{pmatrix}
\]
is a Maurer-Cartan element in $(\aa\coprod_M\bb)^{\oplus 2}$.
One verifies that the following matrix
\[
\Delta=\begin{pmatrix}
0&0&0&0\\
\delta_0&0&0&0\\
\delta_{00}&0&0&0\\
\delta_{01}&\delta_{11}&\delta_1&0
\end{pmatrix}
\text{ acting on }
\begin{pmatrix}
A_0\\B_0\\A_1\\B_1
\end{pmatrix}
\]
defines a Maurer-Cartan element in $(\aa\xrightarrow{M}\bb)^{\oplus 4}$. Rearranging $\Delta$
we get a different Maurer-Cartan element in $(\aa\xrightarrow{M}\bb)^{\oplus 4}$
\[
\Delta^\ast
=\begin{pmatrix}
0&0&0&0\\
\delta_{00}&0&0&0\\
\delta_{0}&0&0&0\\
\delta_{01}&\delta_{1}&\delta_{11}&0
\end{pmatrix}
\text{ acting on }
\begin{pmatrix}
A_0\\A_1\\B_0\\B_1
\end{pmatrix}
\]
which is a block-matrix representation for an object in $(\aa\ast\aa)\coprod_{M^{\ast 2}} (\bb\ast\bb)$.
This construction defines and injection $\Ob((\aa\coprod_M\bb){\ast}(\aa\coprod_M\bb))\hookrightarrow \Ob( \aa{\ast}\aa\coprod_{M^{\ast 2}} \bb{\ast}\bb)$ (but not a bijection) which is is compatible with $\Hom$-sets. It is now
an easy verification (but messy to write down) that we also get compatibility with higher operations.
\end{proof}
The bimodule $M$ may be extended to a $\Free(\bb)$-$\Free(\aa)$-$A_n$-bimodule which we denote by $\Free(M)$. 
\begin{lemma} \label{lem:free}
We have a fully faithful functor of $A_{n-1}$-categories
\[
\textstyle \Free(\aa\coprod_M\bb)\r \Free(\aa)\coprod_{\Free(M)}\Free(\bb).
\]
\end{lemma}
\begin{proof} An object in $\Free(\aa\coprod_M\bb)$ is of the form $\bigoplus_{i\in I} \Sigma^{a_i}(A_i\oplus B_i,\delta_i)$. We send it to
$(\bigoplus_i \Sigma^{a_i}A_i\oplus \bigoplus_i \Sigma^{a_i} B_i,\oplus_i \delta_i)$. It is easy to see that this operation is fully faithful.
\end{proof}
Now we put $\Tw_{\le 1}M=(\Free M)^{\ast 2}$. From Lemma \ref{lem:Aness:star:bimod} we obtain
\begin{lemma} \label{lem:Aness:Tw:bimod}
$\Tw_{\le 1}M$ is a $A_{\lfloor (n-1)/2\rfloor}$-bimodule.
\end{lemma}

\begin{corollary} \label{cor:phiIphiast}
Assume $n\ge 3$. There is a fully faithful functor of 
$A_{\lfloor (n-1)/2\rfloor-1}$-categories
\begin{equation}
\label{eq:tw}
\textstyle \Tw_{\le 1}(\aa\coprod_M\bb)\r \Tw_{\le 1} \aa\coprod_{\Tw_{\le 1}M} \Tw_{\le 1}\bb
\end{equation}
whose restriction to $\aa\coprod_M\bb$ is $(\Phi,I,\Phi^\ast)$ where $\Phi:\aa\r \Tw_{\le 1}\aa$ is as in
Definition \ref{def:truncated}, $\Phi^\ast:\bb\r \Tw_{\le 1}\bb$ is the related map $B\mapsto (0\oplus B,0)$ and $I:M\r {}_{\Phi}\Tw_{\le 1}M_{\Phi^\ast}$ is the obvious inclusion.
\end{corollary}
\begin{proof}
The existence of \eqref{eq:tw} follows by combining Lemma \ref{lem:star} and \ref{lem:free}. The fact that the restriction to $\aa\coprod_M\bb$ has the indicated form follows from the construction of the map.
\end{proof}
\begin{proof}[Proof of Theorem \ref{th:maingluing}]
If $n\ge 13$ then $\Tw_{\le 1}\aa$, $\Tw_{\le 1}\bb$ are at least $A_6$-categories by Lemma \ref{An-ness},
and by Lemma  \ref{lem:Aness:Tw:bimod}  $\Tw_{\le 1}M$ is at least an $A_6$-bimodule.
We can use Corollary \ref{cor:quasi} with $n=6$ and $m=3$, together with Lemma \ref{lem:phiast} below to conclude that the composition 
\[
\textstyle \aa\coprod_M \bb\r \Tw_{\le 1}(\aa\coprod_M\bb)\r \Tw_{\le 1} \aa\coprod_{\Tw_{\le 1}M} \Tw_{\le 1}\bb
\]
(which is equal to $(\Phi,I,\Phi^\ast)$ by Corollary \ref{cor:phiIphiast}) 
is a quasi-equivalence. Since both functors are fully faithful (the second one by Corollary \ref{cor:phiIphiast}), the first one must be a quasi-equivalence as well.

Put $\cc=\aa\coprod_M \bb$. The claim about the exactness of $\varphi_\aa$, $\varphi_\bb$ follows from Theorem \ref{bla}. We clearly also have $H^\ast(\cc)(H^\ast(\bb),H^\ast(\aa))=0$. So
to show that we have a semi-orthogonal decomposition as in \eqref{eq:semi-orth} we have
show that every object $C$ in $H^\ast(\cc)$ is of the form $C\cong \cone(C_\aa\r C_\bb)$ with $C_\aa\in \Ob(\aa)$, $C_\bb\in \Ob(\bb)$. Assume $C=(A\oplus B,\delta)$. We have a fully faithful functor $\aa\coprod_M \bb\subset \Free\aa\coprod_{\Free M}\Free \bb$ and the latter
category is also pre-triangulated (as ``$\Free$'' preserves A-ness). Again by Theorem \ref{bla} this functor is exact.
The following triangle
\[
\Sigma^{-1} A\xrightarrow{\sigma \delta} B\xrightarrow{i} C\xrightarrow[(1)]{p}\Sigma^{-1}A
\]
is distinguished in $H^\ast(\Free\aa\coprod_{\Free M}\Free\bb)$ as it is trivially isomorphic to the standard triangle $\bar{\delta}_{\sigma\delta}$ in $H^\ast(\Tw_{\le 1}(\Free\aa\coprod_{\Free M}\Free\bb))$.
Choose $A'\in \Ob(\aa)$ such that $A'\cong \Sigma^{-1}A$ in $\Free\aa$ ($A'$ is a desuspension of $A$). Then by the axioms of triangulated categories we obtain $\cone(A'\r B)\cong C$
in $H^\ast(\Free\aa\coprod_{\Free M}\Free \bb)$. By fully faithfulness this isomorphism also holds in $H^\ast(\aa\coprod_M\bb)$.

The fact that the corresponding bimodule is as given is clear.
\end{proof}
\begin{lemma} \label{lem:phiast} Let $\aa$ be an $A_n$-category. The strict $A_{\lfloor (n-1)/2\rfloor}$ functors $\Phi,\Phi^\ast :\aa \r \Tw_{\le 1}\aa$ given by $\Phi(A)=(A\oplus 0,0)$, $\Phi^\ast(A)=(0\oplus A,0)$ are homotopic.
\end{lemma}
\begin{proof} The homotopy $h$ is such that $h_n=0$ for $n\ge 1$ and $h_0$ is the matrix $\left(\begin{smallmatrix} 0&1\\1&0\end{smallmatrix}\right)$.
\end{proof}
\section{Higher Toda brackets in triangulated and $A_\infty$-categories}
\subsection{Postnikov systems}
\label{sec:postnikovsystems}
Let 
\begin{equation}
\label{eq:sequence}
X^\bullet:X_0\r X_1\r \cdots\r X_n
\end{equation}
be a complex in a triangulated category $\Tscr$, 
i.e.\ a sequence of composable morphisms in $\Tscr$ such that the composition of any two consecutive morphisms is zero. A \emph{Postnikov system} for $X^\bullet$ is any exact diagram in 
$\Tscr$ of the form
\begin{equation}
\label{diag:postnikov}
\xymatrix@=1.5em{
&Y_0\ar[ddr]&&Y_1\ar[ll]_{(1)}\ar[ddr]&&Y_2\ar[ll]_{(1)}\ar[ddr]
&&&&Y_{n-1}\ar[ddr]&&Y_n\ar[ll]_{(1)}\\
&\circlearrowright&\bold{d}&\circlearrowright&\bold{d}&\circlearrowright&&&&\circlearrowright&\bold{d}\\
X_0\ar@{=}[uur]\ar[rr]&&X_1\ar[uur]\ar[rr]&&X_2\ar[rr]\ar[uur]&&X_3&\cdots&
X_{n-1}\ar[uur]\ar[rr]&& X_n\ar[uur]
}
\end{equation}
where the triangles marked with $\circlearrowright$ are commutative and the triangles marked with $\bold{d}$ are distinguished. This means that we should have the following distinguished triangles
\begin{equation}
\label{eq:dist}
Y_i\r X_{i+1}\r Y_{i+1}\xrightarrow[(1)]{} Y_i
\end{equation}
with $X_0=Y_0$. 
A Postnikov system need not exist and if it exists it may not be unique. If a Postnikov system exists then the object $Y_n$ will be called a \emph{convolution} of $X^\bullet$.
\begin{remark} \label{lem:intuition}
 Sometimes it is helpful to think of a convolution $Y_n$ as an object with an ascending filtration with subquotients (starting from the bottom) $X_n,\Sigma X_{n-1}, \Sigma^2 X_{n-2},\ldots,\Sigma^{n}X_0$.
In particular the convolution $Y_n$ comes with maps
\begin{equation}
\label{eq:withmaps1}
\xymatrix{
&Y_n\ar[dl]_p^{(n)}&\\
X_0&&X_n\ar[ul]_i
}
\end{equation}
where $i$ is as \eqref{diag:postnikov} and $p$ it the composition $Y_n\r Y_{n-1}\r\cdots\r Y_0=X_0$ in that same diagram. Note that $pi=0$.
\end{remark}
\subsection{Existence}
Some existence and functoriality results for Postnikov systems are stated in  \cite[Lemmas 1.5, 1.6]{Orlov4} but since they require the vanishing of arbitrary negatives $\Ext$'s between
suitable objects, they are not completely sufficient for our purposes. So we give some slightly strengthened versions in the next two sections.
\begin{lemma} Assume $X^\bullet$ is a complex in a triangulated category $\Tscr$ such that
\begin{equation}
\label{eq:existence}
\Tscr(X_a,X_b)_{-(b-a)+2}=0\qquad \text{ for $b\ge a+3$}.
\end{equation}
Then $X^\bullet$ may be extended to a Postikov system. Moreover if the following condition holds
\begin{equation}
\label{eq:uniqueness}
\Tscr(X_a,X_b)_{-(b-a)+1}=0\qquad \text{for $b\ge a+2$}
\end{equation}
then such an extension is unique, up to non-unique isomorphism.
\end{lemma}
\begin{proof}
The Posnikov system built on $X^\bullet$ will be constructed inductively. Assume we
have constructed the part involving $X_0,X_1,\ldots,X_i$, $Y_0,Y_1,\ldots, Y_{i}$ (so this is a Postnikov system on $X_0\r\cdots\r X_i$). To lift the map $X_i\r X_{i+1}$ to a map $Y_i\r X_{i+1}$ we need that the
composition
\[
Y_{i-1}\r X_i\r X_{i+1}
\]
is zero. Since the composition of $X_{i-1}\r X_i\r X_{i+1}$ is zero by definition
it follows from \eqref{eq:dist} that we should have $\Tscr(Y_{i-2},X_{i+1})_{-1}=0$. Using Remark \ref{lem:intuition} we see that this condition is implied by \eqref{eq:existence}.

Once we have lifted to $X_i\r X_{i+1}$ to $Y_i\r X_{i+1}$ we may construct $Y_{i+1}$ via the the distinguished triangle \eqref{eq:dist}.

\medskip

To obtain uniqueness we note that if $X^\bullet$ can be extended to two Postnikov systems then by Lemma \ref{th:postnikov} below the identity on $X^\bullet$ can be extended to
a morphism between these Postnikov systems. It is then easy to see that this extension must be an isomorphism.
\end{proof}
\subsection{Weak functoriality}
\begin{lemma}
\label{th:postnikov}
Assume we have a morphism of complexes in a triangulated category $\Tscr$
\begin{equation}
\label{postnikov}
\xymatrix{
X_0\ar[rr]\ar[d]&& X_1\ar[rr]\ar[d]&&X_2\ar[r]\ar[d]& \cdots\ar[r]& X_n\ar[d]\\
X'_0\ar[rr]&& X'_1\ar[rr]&&X'_2\ar[r]& \cdots\ar[r]& X'_n\\
}
\end{equation}
such that the following conditions hold
\begin{align}
\Tscr(X_a,X'_b)_{-(b-a)+1}&=0\qquad \text{for $b\ge a+2$} \label{eq:functoriality3}
\end{align}
Then \eqref{postnikov} can be extended to a map of Postnikov
systems (not necessarily uniquely).
\end{lemma}
\begin{proof}
We work inductively. 
Assume that we have defined the extended map
on $Y_0,\ldots,Y_i$ with the required commutativity holding
on $Y_0,\ldots,Y_i,X_0,\ldots,X_i$.
We perform the induction step. We have a diagram
\begin{equation}
\label{eq:morf}
\xymatrix{
Y_i\ar[r]\ar[d]\ar@{.>}[dr]|{\delta}&X_{i+1}\ar[r]\ar[d]&Y_{i+1}\ar[r]^{(1)}&Y_i\ar[d]\\
Y'_i\ar[r]&X'_{i+1}\ar[r]&Y'_{i+1}\ar[r]_{(1)}&Y'_i\\
}
\end{equation}
We do not know that the left most square is commutative, so  let the dotted arrow denote the difference of the two compositions. From the following diagram 
\[
\xymatrix{
X_i\ar[r]\ar[d]&Y_i\ar@{.>}[rd]|{\delta}\ar[r]\ar[d]&X_{i+1}\ar[d]\\
X'_i\ar[r]&Y'_i\ar[r]&X'_{i+1}
}
\]
we obtain that the composition of $\delta$ with $X_i\r Y_i$ is zero. So in view
of the distinguished triangle 
\[
Y_{i-1}\r X_i\r Y_i\xrightarrow[(1)]{} Y_{i-1}
\]
$\delta$ will be zero provided $\Tscr(Y_{i-1},X'_{i+1})_{-1}=0$.  This follows from Remark \ref{lem:intuition} and  the hypothesis \eqref{eq:functoriality3}.

So $\delta=0$ and the square in 
\eqref{eq:morf} is commutative. We now finish by invoking TR3.
\end{proof}
\subsection{Higher Toda brackets}
\label{sec:higher:toda}
In this section we define higher Toda brackets. One may verify that they are the same as those defined in \cite{ChristensenFrankland}.
\begin{definition}\label{def:toda} Let $X^\bullet=((X_i)_{i=0}^{n},(d_i)_{i=0}^{n-1})$ for $n\ge 3$ be a complex in a triangulated category $\Tscr$.
The (higher) \emph{Toda bracket} $\langle X^\bullet \rangle \subset \Tscr(X_{0},X_{n})_{-n+2}$ of $X^\bullet$ is 
the collection of compositions $\beta\alpha$ where $\alpha$, $\beta$ fit in the following commutative diagram
\begin{equation}
\label{diag:toda}
\xymatrix{
&&&Y\ar[dl]_p^{(n-2)}\ar@/^1pc/@{.>}[rrrd]^\beta\\
X_{0}\ar@/^1pc/@{.>}[rrru]^\alpha_{(-n+2)}\ar[rr]_{d_0} && X_1&& X_{n-1}\ar[rr]_{d_{n-1}}\ar[ul]_i && X_{n}
}
\end{equation}
where $Y$ is a convolution of $(X_i)_{i=1}^{n-1}$ and $p$, $i$ are as in \eqref{eq:withmaps1}.

\end{definition}
 Note that
if $n>3$ then $\langle X^\bullet \rangle$ may be empty.
\begin{theorem}
Let $X^\bullet$ be as in Definition \ref{def:toda}. 
\begin{enumerate} 
\item If $t\in \langle X^\bullet \rangle$ then $t+d_{n-1}\Tscr(X_0,X_{n-1})_{-n+2}+\Tscr(X_1,X_n)_{-n+2}d_0\subset
\langle X^\bullet \rangle$. 
\item If
\begin{equation}
\label{eq:todaexistence}
\begin{aligned}
\Tscr(X_a,X_b)_{-(b-a)+2}=0\qquad \text{for $b-a\in [3,n-1]$}.
\end{aligned}
\end{equation}
then $\langle X^\bullet\rangle\neq \emptyset$.
\item If moreover
\begin{equation}
\label{toda:coset}
\Tscr(X_a,X_b)_{-(b-a)+1}=0\qquad \text{for $b-a\in [2, n-2]$}
\end{equation}
then $\langle X^\bullet \rangle$ is a coset of $d_{n-1}\Tscr(X_0,X_{n-1})_{-n+2}+\Tscr(X_1,X_n)_{-n+2}d_0$.
\end{enumerate}
\end{theorem}
\begin{proof}
\begin{enumerate}
\item If $\phi \in \Tscr(X_0,X_{n-1})_{-n+2}$ then as $pi\phi=0$, adding to $i\phi$ to $\alpha$ still keeps the diagram \eqref{diag:toda} commutative. Since $\beta i\phi=d_{n-1}\phi$ we obtain
that $t+d_{n-1}\phi\in \langle X^\bullet\rangle$. A similar reasoning applies if
we start with $\phi\in \Tscr(X_1,X_n)_{-n+2}$.
\item Note that \eqref{eq:todaexistence} implies in particular \eqref{eq:existence} for $(X_i)_{i=1}^{n-1}$. So a convolution $Y$ as in \eqref{diag:toda} exists and we have to verify the existence of $\alpha$ and $\beta$.
We will now introduce notations similar to \S\ref{sec:postnikovsystems}. So we will denote the Postnikov systems giving rise to $Y$ by $Y_1,\ldots,Y_{n-1}$ where $Y_{n-1}=Y$ and $Y_1=X_1$.

We first consider the existence of $\beta$.
We have a distinguished triangle
\begin{equation}
\label{eq:firstdist}
Y_{n-2}\r X_{n-1}\xrightarrow{i} Y_{n-1}\r
\end{equation}
Thus in order for the map $d_{n-1}:X_{n-1}\r X_n$ to factor through $Y_{n-1}$ we have to prove that  the composition $Y_{n-2}\r X_{n-1} \r X_n$ is zero. Since we already know that the composition
$X_{n-2}\r Y_{n-2}\r X_{n-1}\r X_n$ is zero and there is a distinguished triangle
\[
X_{n-2}\r Y_{n-2}\r \Sigma Y_{n-3}\r 
\]
it is sufficient to show that $\Tscr(\Sigma Y_{n-3}, X_n)_0=0$. Now by Remark \ref{lem:intuition}, $\Sigma Y_{n-3}$ has subquotients $\Sigma X_{n-3},\allowbreak\dots,\allowbreak\Sigma^{n-3} X_1$. The conclusion now follows
from \eqref{eq:todaexistence}.

Now we look at the existence of $\alpha$.
We will successively lift $X_0\xrightarrow{d_0} X_1=Y_1$  to maps $X_0\xrightarrow{(-1)} Y_2$, \dots. $X_0\xrightarrow{(-n+2)} Y_{n-1}$. The last map is the sought $\alpha$. 
First we look at the distinguished triangle
\[
X_1\r X_2\r Y_2\r 
\]
Since the composition $X_0\r X_1\r X_2$ is zero the map $d_0$ factors through $\Sigma^{-1}Y_2$. To continue we use the distinguished triangles
\[
Y_{i-1}\r X_i\r Y_{i}\r 
\]
for $3\le i\le n-1$.
Assume we have constructed the map $X_0\r \Sigma^{-i+2} Y_{i-1}$. From \eqref{eq:todaexistence}. 
we obtain that the composition $X_0\r \Sigma^{-i+2}Y_{i-1}\r \Sigma^{-i+2}X_i$ is zero and
hence $X_0\r \Sigma^{-i+2} Y_{i-1}$ factors through $\Sigma^{-i+1} Y_i$ and we can continue.
\item  First we observe that \eqref{toda:coset} implies in particular \eqref{eq:uniqueness} and hence
the Postnikov system built on $(X_i)_1^{n-1}$ is unique.
To prove asserted statement we have to investigate the freedom in choosing $\alpha$ and $\beta$.

Again we will discuss $\beta$ first. $\beta$ is determined up to an element of the kernel of $\Tscr(Y_{n-1},X_n)_0\r\allowbreak \Tscr(X_{n-1},X_n)_0$. Using
the distinguished triangle \eqref{eq:firstdist}
we see that $\beta$ is determined up to a  composition of the form $Y_{n-1}\r \Sigma Y_{n-2}
\xrightarrow{\gamma} X_n$. Using Remark \ref{lem:intuition} we see that $\Sigma Y_{n-2}$ 
has subquotients $\Sigma X_{n-2},\Sigma^2 X_{n-3},\ldots \Sigma^{n-2} X_1$. Hence by  \eqref{toda:coset} any morphism
$\Sigma Y_{n-2}\r X_n$ factors through $\Sigma^{n-2} X_1$. It follows that $\beta$ is
determined up to a composition of the form $Y_{n-1}\xrightarrow{p} \Sigma^{n-2}X_1\xrightarrow{\gamma'} X_n$. Composing with $\alpha$ we see  as in (1) that changing $\beta$ in this way,
changes $\beta\alpha$ by an element of $\Tscr(X_1,X_n)_{-n+2}d_0$.

Now we discuss $\alpha$. $\alpha$ is determined up to an element of $\ker (\Tscr(X_0,Y_{n-1})_{-n+2}
\r \Tscr(X_0,X_1)_0$. 
Define $Y'_i=\Sigma^{-1}\cone (Y_i\r \Sigma^{i-1}X_1)$, so that in particular $Y'_1=0$, $Y'_2=X_2$. Using the octahedral axiom we may construct
commutative diagrams for $i=2,\ldots, n-1$
\[
\xymatrix{
&&&\\
\Sigma Y'_{i-1} \ar[u]\ar[r] & \Sigma Y_{i-1}\ar[r]\ar[u]&\Sigma^{i-1} X_1\ar@{=}[d]\ar[r] &
\\
Y'_{i}\ar[u] \ar[r] & Y_{i}\ar[r]\ar[u]&\Sigma^{i-1} X_1\ar[r] &
\\
X_i\ar[u]\ar@{=}[r] &X_i\ar[u]\\
}
\]
with rows and columns that are distinguished triangles, where the maps not involving $Y'$'s are taken from the Postnikov system.  Hence similar to Remark \ref{lem:intuition}, $Y'_{i}$ has subquotients $X_{i},\Sigma X_{i-1},\ldots,\allowbreak\Sigma^{i-2}X_2$.

We have a distinguished triangle
\[
Y'_{n-1}\r Y_{n-1}\r \Sigma^{n-2} X_1\r 
\]
and hence  $\alpha$ is determined up to a composition $X_0\xrightarrow{\delta} \Sigma^{-n+2}
Y'_{n-1}\r \Sigma^{-n+2}Y_{n-1}$. Now  $ \Sigma^{-n+2}
Y'_{n-1}$ has subquotients $\Sigma^{-n+2}X_{n-1},\ldots,\Sigma^{-1}X_2$ and hence 
by \eqref{toda:coset} we obtain that any map $X_0\xrightarrow{\delta} \Sigma^{-n+2}Y'_{n-1}$  factors through $\Sigma^{-n+2}X_{n-1}$. Hence we obtain that $\alpha$ is determined up
to a composition $X_0\xrightarrow{\delta'} \Sigma^{-n+2}X_{n-1}\r \Sigma^{-n+2} Y'_{n-1}\r \Sigma^{-n+2} Y_{n-1}$ which by construction is the same as a composition
$X_0\xrightarrow{\delta'} \Sigma^{-n+2}X_{n-1}\xrightarrow{\Sigma^{-n+2}i} Y_{n-1}$. We now finish as for $\beta$.\qedhere
\end{enumerate}
\end{proof}
\subsection{Postnikov systems associated to twisted complexes}
In this section $\aa$ is an $A_\infty$-category.
\subsubsection{More on the category $\Free(\aa)$}
Recall that in \S\ref{sec:free} we introduced the strict endo-functor  $\Sigma$ of $\Free(\aa)$. 
Below we introduce some more notation concerning the category $\Free(\aa)$. If $X=\Sigma^x Y$ for $Y\in \Ob(\aa)$ then we let $\eta_{X,a,b}:
\Sigma^aX\r \Sigma^b X$ be given by $\sigma^{b-a}\Id_Y$. Note that $m_2(\eta_{X,b,c},\eta_{X,a,b})=\eta_{X,a,c}$. All operations
on $\aa$, except $m_2$, vanish when one of its arguments is of the form $\eta_{X,b,c}$. 
Moreover we have formulas
\begin{align*}
m_n(\dots,f,m_2(\eta_{X,a,b},g),\ldots)&=m_n(\dots,m_2(f,\eta_{X,a,b}),g,\ldots)\\
m_n(m_2(\eta_{X,a,b},f),\dots)&=(-1)^{(-2+n)(b-a)}m_2(\eta_{X,a,b},m_n(f,\ldots))\\
m_n(\dots,m_2(g,\eta_{X,a,b})))&=m_2(m_n(\ldots,g),\eta_{X,a,b}).
\end{align*}
and their $b$-versions which are useful for computations
\begin{equation}
\label{eq:useful}
\begin{aligned}
b_n(\dots,sf,sm_2(\eta_{X,a,b},g),\ldots)&=(-1)^{b-a}b_n(\dots,sm_2(f,\eta_{X,a,b}),sg,\ldots)\\
b_n(sm_2(\eta_{X,a,b},f),\dots)&=m_2(\eta_{X,a,b},b_n(sf,\ldots))\\
b_n(\dots,sm_2(g,\eta_{X,a,b})))&=m_2(b_n(\ldots,sg),\eta_{X,a,b}).
\end{aligned}
\end{equation}
Below we usually write $\eta_{X,a,b}g$ for $m_2(\eta_{X,a,b},g)$ and similarly $m_2(g,\eta_{X,a,b})$. By the vanishing of $m_3$ on arguments involving $\eta_{X,a,b}$
this will not lead to confusion. Sometimes
we also write $\eta^{-1}_{X,a,b}$ for $\eta_{X,b,a}$. One verifies using
the definition of the functor $\Sigma$ (see \S\ref{sec:free})
that for $f:\Sigma^a X\r \Sigma^b Z$ one has
\begin{equation}
\label{eq:sigmadef}
\Sigma^nf=(-1)^{n|f|}\eta_{Z,b,b+n}f\eta_{X,a+n,a}.
\end{equation}
Finally we put $\eta_{X}=\eta_{X,0,1}$.
\subsubsection{More on the triangulated structure of $\Tw \aa$}
Let $f:(A,\delta_A)\r (B,\delta_B)$ be a closed morphism of degree $0$ in $\Tw \aa$. To $f$ we associate a triangle in $H^\ast(\Tw \aa)$.
\begin{equation}
\label{eq:standard1}
(A,\delta_A) \xrightarrow{f} (B,\delta_B)\xrightarrow{i} (C(f),\delta_{C(f)}) \xrightarrow[(1)]{p} (A,\delta_A)
\end{equation}
where $C(f)=\Sigma A\oplus B$ and 
\[
\delta_{C(f)}=\begin{pmatrix}
\Sigma \delta_A &0\\
f\eta_A^{-1}&\delta_B
\end{pmatrix}
\]
and furthermore
\begin{equation}
\label{eq:iptwisted}
i=\begin{pmatrix}0\\\id_B
\end{pmatrix},\qquad 
p=(\eta_A^{-1}\ 0).
\end{equation}
The following lemma is an easy verification:
\begin{lemma} The triangles \eqref{eq:standard1} are distinguished according to Definition \ref{triangulatedstructure}.
\end{lemma}
\subsubsection{Postnikov systems from objects in $\Tw \aa$}
\begin{proposition}
\label{prop:postnikovtwisted}
A twisted complex in $\Tw \aa$
\[
Y_n=(\Sigma^nX_0\oplus \Sigma^{n-1}X_1\oplus\cdots\oplus X_n,\delta)
\]
with $X_i\in \Free(\aa)$ gives rise to a Postnikov system in $H^\ast(\Tw\aa)$ built on the complex
\[
X_0\xrightarrow{d_0}X_1\xrightarrow{d_1}\cdots\xrightarrow{d_{n-1}}X_n
\]
with 
\[
d_{j-1}=(-1)^{n-j} \eta_{X_j,n-j,0}\cdot\delta_{j,j-1} \cdot\eta_{X_{j-1},0,n-j+1},
\]
where $\delta_{j,j-1}:\Sigma^{n-j+1} X_{j-1}\r \Sigma^{n-j}X_j$ is the $(j,j-1)$ entry of the matrix $\delta$ (the $\cdot$'s are for easier reading).

In the Postnikov system we also have
\[
Y_j=(\Sigma^{j}X_0\oplus \Sigma^{j-1}X_1\oplus\cdots\oplus X_j,\delta_{Y_j})
\]
such that $\Sigma^{n-j}\delta_{Y_i}$ is given by the upper left $j+1\times j+1$-square in the matrix representing $\delta$. 

Finally the maps $p:Y_n\xrightarrow{(n)} X_0$, $i:X_n\r Y_n$ as in \eqref{eq:withmaps1} are given by
\begin{equation}
\label{eq:twisted2}
i=\begin{pmatrix} 0\\\vdots\\\Id_{X_n}\end{pmatrix}\qquad p=(\eta_{X_0,n,0},0,\ldots,0)
\end{equation}
\end{proposition}
\begin{proof}
 We may write
\[
\delta_{Y_j}=
\begin{pmatrix}
\Sigma \delta_{Y_{j-1}}&0\\
f_j\eta_{Y_{j-1}}^{-1}&0
\end{pmatrix}
\]
where $f_j:(Y_{j-1},\delta_{Y_{j-1}})\r X_j$ is the closed map in $\Tw\aa$
with matrix
\[
((\delta_{Y_j})_{j,0}\eta_{\Sigma^{j-1} X_0},\ldots,(\delta_{Y_j})_{j,j-1} \eta_{X_{j-1}}).
\]

Clearly $Y_j=C(f_j)$ so that we have standard triangles
\begin{equation}
\label{eq:standardpostnikov}
Y_{j-1}\xrightarrow{f_j} X_j\xrightarrow{i_j} Y_j\xrightarrow[(1)]{p_j}Y_{j-1}
\end{equation}
where $(i_j,p_j)$ are as in \eqref{eq:iptwisted}. In particular $i=i_n$ is given by the formula \eqref{eq:twisted2}.
We compute the composition
\[
X_{j-1}\xrightarrow{i_{j-1}} Y_{j-1}\xrightarrow{f_j} X_j
\]
It is given by the matrix multiplications

\[
((\delta_{Y_j})_{j,0}\eta_{\Sigma^{j-1} X_0},\ldots,(\delta_{Y_j})_{j,j-1} \eta_{X_{j-1}})
\begin{pmatrix}
0\\
\vdots\\
0\\
\Id_{X_{j-1}}
\end{pmatrix}
=
(\delta_{Y_j})_{j,j-1}\eta_{X_{j-1}}
\]
which is equal to
$
(\Sigma^{-(n-j)} (\delta_{Y_n})_{j,j-1})\eta_{X_{j-1}}
$. One computes using \eqref{eq:sigmadef} that the latter expression is equal to $d_{j-1}$.

Finally to show $p$ is as in \eqref{eq:twisted2} we use
$
p=p_1\cdots p_{n-1}p_n
$
 by the description in Remark \ref{lem:intuition}. Then we use the formula \eqref{eq:iptwisted} for $p_j$.
\end{proof}
\subsection{Higher Toda brackets in $A_\infty$-categories}
We prove the following result.
\begin{theorem} \label{th:toda:Ainfty}
Let $\aa$ be a pre-triangulated $A_\infty$-category and let $X_0\xrightarrow{d_0}X_1\xrightarrow{d_1}\cdots\xrightarrow{d_{n-1}}X_n$
be a complex in $\Tscr=H^\ast(\aa)$. Assume the following conditions hold:
\begin{enumerate}
\item The $A_\infty$-subcategory of $\aa$ spanned by the objects $(X_i)_i$ is minimal (i.e.\ $b_1=0$).
\item 
$
\aa(X_i,X_j)_u=0 \qquad \text{for $-n+2<u<0$}
$.
\end{enumerate}
Using (1) we may regard $d_i$ as closed arrows in $\aa$. With this convention we have that
$\langle X^\bullet \rangle$ is the coset for $d_{n-1}\Tscr(X_0,X_{n-1})_{-n+2}+\Tscr(X_1,X_n)_{-n+2}d_0$
given by $\overline{s^{-1}b_n(sd_{n-1},\ldots,sd_0)}$.
\end{theorem}
\begin{proof}
Since \eqref{eq:todaexistence} and \eqref{toda:coset} hold it is sufficient to produce a single element of $\langle X^\bullet\rangle$.
Since higher Toda brackets are obvously invariant under equivalences of triangulated categories we may perform the calculation in $\Tw\aa$.
We start with the Postnikov system built on $X_1\xrightarrow{d_1}X_2\xrightarrow{d_2}\cdots \xrightarrow{d_{n-2}} X_{n-1}$.  By 
Proposition \ref{prop:postnikovtwisted} it is obtained from the twisted complex
\[
Y=(\Sigma^{n-2}X_1\oplus\cdots\oplus X_{n-1},\delta)
\]
where the only non-zero entries of $\delta$ are $\delta_{j,j-1}$ for $j=2,\ldots,n-1$ and $\delta_{j,j-1}$ is given by
\[
 \delta_{j,j-1}=(-1)^{n-1-j} \eta_{X_j,0,n-1-j}\cdot d_{j-1} \cdot\eta_{X_{j-1},n-j,0} 
\]
Using the formulas for $i$ and $p$ (see \eqref{eq:twisted2}) it is then easy to see that we may take
\[
\alpha=\begin{pmatrix} \eta_{X_1,0,n-2}d_0\\\vdots\\0\end{pmatrix}\qquad  \beta=(0,0,\ldots,d_{n-1})
\]
Then $\overline{m_{\Tw\aa,2}(\beta,\alpha)}\in \langle X^\bullet\rangle$. It will be more convenient to compute $b_{\Tw\aa,2}(s\beta,s\alpha)=sm_{\Tw\aa,2}(\beta,\alpha)$. We have
\begin{align*}
  b_{\Tw\aa,2}(s\beta,s\alpha)&=b_{\aa,n}(s\beta,\underbrace{s\delta,\ldots,s\delta}_{n-2},s\alpha)\\
&=b_{\aa,n}(sd_{n-1},sd_{n-2} \cdot\eta_{X_{n-2},1,0},
\ldots,(-1)^{n-1-j}s\eta_{X_j,0,n-1-j}\cdot d_{j-1} \cdot\eta_{X_{j-1},n-j,0} ,\ldots,\\
&(-1)^{n-3}s\eta_{X_2,0,n-3}\cdot d_{1} \cdot\eta_{X_{1},n-2,0} ,
s\eta_{X_1,0,n-2}d_0)\\
&=b_{\aa,n}(sd_{n-1},sd_{n-2},\ldots, s d_{1} ,sd_0)
\end{align*}
where in the last line we have used \eqref{eq:useful}.
\end{proof}

\section{Triangulated categories without models}
\label{sec:nomodel}
If $\Ascr$ is a triangulated category then an $A_\infty$-enhancement on $\Ascr$ is a pair consisting of a pre-triangulated $A_\infty$-category $\aa$ 
such that $\Ob(\aa)=\Ob(\Ascr)$ and an isomorphism of triangulated categories
$H^\ast(\aa)\r \Ascr$ inducing the identity on objects.
The following proposition will be the basis for constructing a triangulated category that does not admit an $A_\infty$-enhancement.
\begin{proposition}\label{prop:no lift if strong unique enhancement}
Let $\aa$, $\bb$ be pre-triangulated $A_{\infty}$-categories. Suppose we have an $A_n$- functor $F:\aa\to \bb$ for $n\ge 13$ such that 
$H^\ast(F)$ does not lift to an $A_{\infty}$-functor for any $A_\infty$-enhancements on $H^\ast(\aa)$, $H^\ast(\bb)$.
Let $\cc$ be the gluing category $\cc=\aa \coprod_M \bb$ where $M={}_F\bb$
(see \S\ref{sec:bimodules}).
Then $H^\ast(\cc)$ is a triangulated category which does not admit an $A_{\infty}$-enhancement.
\end{proposition}

\begin{proof}
By the discussion in \S\ref{sec:bimodules}, $M$ is an $A_{14}$-bimodule. Therefore by Theorem \ref{th:maingluing}, $\cc$ is a pre-triangulated $A_{13}$-category. Hence
by Theorem \ref{th:mainth}, $H^\ast(\cc)$ is triangulated.

  Suppose that an $A_\infty$-enhancement $\dd$ on $H^\ast(\cc)$ exists. Since $H^\ast(\aa)$, $H^\ast(\bb)$
  are full exact subcategories of $H^\ast(\cc)$ (see Theorem \ref{th:maingluing}), it follows that the
  $A_{\infty}$-structure on $\dd$ induces $A_{\infty}$-enhancements  $\aa'$, $\bb'$,
 on 
   $H^\ast(\aa)$ and $H^\ast(\bb)$. By
$H^\ast(\dd)\cong H^\ast(\cc)$
it follows that $\dd(A,B)_{A,B}$ for $A\in \Ob(H^\ast(\aa))$, $B\in \Ob(H^\ast(\bb))$ defines a $A_\infty$-$\bb'$-$\aa'$-bimodule which is a co-quasi-functor in the sense of \S\ref{sec:coquasi} below.
Hence by Lemma \ref{lem:Ainftylift}, $\dd$ induces an $A_\infty$-functor
  $F': \aa' \to \bb'$ such that $H^\ast(F')\cong F$. This contradicts the hypotheses on~$F$.
\end{proof}
\begin{remark} The idea of creating a triangulated category without model by gluing a non-enhanceable functor was suggested to us by Bondal and Orlov on a number of occasions. 
In fact, the idea of translating an enhancement of the glued category into a $A_\infty$-enhancement of the gluing functor, thereby obtaining a contradiction, was specifically suggested to us by Orlov.
\end{remark}
\subsection{Co-quasi-functors}
\label{sec:coquasi}
To fill in a missing ingredient in the proof of Proposition \ref{prop:no lift if strong unique enhancement}
we use an $A_\infty$-version
of the notion of a (co)-quasi-functor (see \cite{Keller1}).
In the rest of this section we assume that $\aa$, $\bb$ are $A_\infty$-categories.
\begin{definition} An $A_\infty$-$\bb$-$\aa$-bimodule $M$ is a \emph{co-quasi-functor} $\aa\r \bb$ if for every object
$A\in \Ob(\aa)$ there exists  $f A\in \Ob(\bb)$ together with an element
 $\bar{\phi}_A\in (H^\ast M)(A,fA)_0$ inducing an isomorphism for all $B\in \Ob(\bb)$:
$
\tilde{\phi}_A:H^\ast(\bb)(fA,B)\r (H^\ast M)(A,B):u\mapsto u\bar{\phi}_A 
$.
\end{definition}
It is clear from the definition that being a co-quasi-functor depends only on the structure of $H^\ast M$ as graded
$H^\ast(\aa)-H^\ast(\bb)$-bimodule. A co-quasi-functor induces an actual functor $f^\circ:H^\ast(\aa)\r H^\ast(\bb)$.
Indeed for $u:A\r A'$ in $H^\ast(\aa)$, $f^\circ u:fA\r fA'$ is defined to be the unique morphism such that
$H^\ast(\bb)(f^\circ u,-)$ is the composition 
$H^\ast (\bb)(fA',-)\xrightarrow[\cong]{\tilde{\phi}_{A'}} (H^\ast M)(A',-)\xrightarrow{H^\ast M(f,-)} 
(H^\ast M)(A,-)\xrightarrow[\cong]{\tilde{\phi}^{-1}_A} H^\ast(\bb)(fA,-)$. Moreover is is clear that different choices of $(\phi_A,f^\circ A)$ lead to 
naturally isomorphic functors.
\begin{lemma} \label{lem:Ainftylift}
Assume that $M$ is a co-quasi-functor $\aa\r \bb$
and let $f^\circ:H^\ast(\aa)\r H^\ast(\bb)$ be the induced functor as explained above. Then there exists an $A_\infty$-functor
$f:\aa\r \bb$ such that $H^\ast(f)=f^\circ$.
\end{lemma}
\begin{proof}
Let $\Cscr^l_\infty(\bb)$ be the DG-category of strictly unital left $A_\infty$-$\bb$-modules \cite[Chapitre 5]{Lefevre}
and let $Y:\bb\r  \Cscr^l_\infty(\bb)^\circ:B\mapsto \bb(B,-)$ be the Yoneda embedding.  Furthermore let $\tilde{\bb}\subset \Cscr^l_\infty(\bb)^\circ$ be the full subcategory spanned by $A_\infty$-modules $M$ which are $A_\infty$-quasi-isomorphic to some $\bb$-module
of the form $\bb(B,-)$.
 Clearly we have that $Y$ corestricts to an $A_\infty$-quasi-equivalence $Y^c:\bb\r \tilde{\bb}$. 
Since $M$ is a co-quasi-functor the
 image of the $A_\infty$-functor $F:\aa\r \Cscr^l_\infty(\bb)^\circ:A\mapsto M(A,-)$ lies in $\tilde{\bb}$. Let $F^c:\aa\r\tilde{\bb}$ be the corestriction of $F$.

Choose
an $A_\infty$-quasi-inverse $W:\tilde{\bb} \r \bb$ to the quasi-equivalence $Y^c:\bb\r \tilde{\bb}$ which sends $M(A,-)$ to $fA$ for $A\in \Ob \aa$ and $u$ to (a representative of) $f^\circ u$ for $u:A\r A'$ a closed map in $\aa$. By Lemma \ref{lem:quasi-inverse}
this is possible.
Then one
easily verifies that $f^\circ=H^\ast(F^c W)$.
\end{proof}
\begin{remark}
It is also easy to prove that we have an quasi-isomorphism of $A_\infty$-bimodules ${}_f\bb\cong M$. However we will not need this.
\end{remark}
\subsection{Localization of triangulated categories}
\label{sec:loc}
The following result is well-known, although we did not find the precise statement we require.
Since the proof is short we include it for the convenience of the reader. 
\begin{proposition} 
\label{prop:loc} Let $\Tscr$ be a triangulated category admitting
  arbitrary coproducts and let $T\in \Ob(\Tscr)$ be a compact
  generator for $\Tscr$. Let $S\subset \Tscr(T,T)$ be a graded right
  Ore set and let $\Tscr_S$ be the full subcategory of $\Tscr$ spanned by the objects $X$ such that $\Tscr(s,X)$ is an isomorphism for all $s\in S$, or equivalently the objects for which
\begin{equation}
\label{eq:ore1}
\Tscr(T,X)\r \Tscr(T,X)_S
\end{equation}
is an isomorphism.
Then $\Tscr_S$ is a triangulated
subcategory of $\Tscr$ and moreover the inclusion functor $\Tscr_S\r \Tscr$ has a left adjoint, denoted by $(-)_S$ such that for $Y\in \Ob(\Tscr)$  the induced map
\[
(-)_S:\Tscr(T,Y)\r \Tscr_S(T_S,Y_S)
\]
factors uniquely through an isomorphism 
\begin{equation}
\label{eq:factor}
\Tscr(T,Y)_S\cong \Tscr_S(T_S,Y_S).
\end{equation}
\end{proposition}
\begin{proof} The fact that $\Tscr_S$ is triangulated follows trivially from the 5-lemma.
Let us now discuss the existence of the adjoint. Let $\Cscr$ be the full subcategory of
$\Tscr$ spanned by objects $X$ such that all morphisms $T\r X$ (not necessarily of degree zero)
are annihilated after composing with some $s:T\r T\in S$, or equivalently
\begin{equation}
\label{eq:ore2}
\Tscr(T,X)_S=0
\end{equation}
It is clear that $\Cscr$ is triangulated and closed under arbitrary coproducts (the latter by the compactness of $T$).

For $s\in S$ let $C(s)$ be the cone of the morphism $s:T\r \Sigma^{|s|}T$. It is clear
that $\Tscr_S=\langle C(s)_{s\in S}\rangle^\perp$.
By the Ore condition on $S$ the objects $C(s)$ are in $\Cscr$. Moreover as 
$\langle C(s)_{s\in S}\rangle^\perp\cap \Cscr=\Tscr_S\cap \Cscr$ and it is easy to see that $
\Tscr_S\cap \Cscr=0$, we obtain that $\Cscr$ is in fact generated by $\langle C(s)_{s\in S}\rangle$. This
yields $\Cscr^\perp=\Tscr_S$.

Hence in particular $\Cscr$ is compactly generated
and using the Brown representability theorem we obtain
that the inclusion functor $\Cscr\r \Tscr$ has a right adjoint $U:\Tscr\r \Cscr$ such 
that every $X\in \Tscr$ fits in a unique distinguished triangle
\begin{equation}
\label{eq:locsec}
UX\r X\r VX\r 
\end{equation}
where $VX\in \Cscr^\perp=\Tscr_S$. It follows easily that $X\r VX$ is a functor $\Tscr\r\Cscr^\perp=\Tscr_S$. Applying $\Tscr(-,Z)$ for $Z\in \Tscr_S$ to \eqref{eq:locsec} we obtain that
$V$ is the sought left adjoint $(-)_S$ to the inclusion $\Tscr_S\r \Tscr$.

Finally we discuss the formula \eqref{eq:factor}. As $\cone(Y\r Y_S)=\Sigma UY\in \Cscr$ we have $\Tscr(T,\cone(Y\r Y_S))_S=0$ by \eqref{eq:ore2}. Hence
$(-)_S$ induces an isomorphism $\Tscr(T,Y)_S\xrightarrow{\cong} \Tscr(T,Y_S)_S\overset{\eqref{eq:ore1}}{=}\Tscr(T,Y_S)=\Tscr_S(T_S,Y_S)$ where the last equality is adjointness.
\end{proof}
\subsection{A non-enhanceble functor}
\label{sec:noenhance}
Now let $k$ be either a field of characteristic zero or an infinite field of characteristic $>n\ge 3$.
Put $R=k[x_1,\ldots,x_n]$ and let $K$ be the quotient field of $R$. 
Furthermore let $R[\varepsilon]$ be the $R$-linear DG-algebra with $|\varepsilon|=-n+2$, $\varepsilon^2=0$, $d\varepsilon=0$. 
Let $C(R,R)$ be the Hochschild complex of $R$ and  let $\HH^n(R,R)=H^n(C(R,R))$. Let $T^n_{R/k}=\wedge^n_R\Der_k(R,R)$.
The HKR theorem gives an inclusion $T^n_{R/k}\subset Z^n C(R,R)$ which induces an isomorphism $T^n_{R/k}\cong \HH^n(K,K)$.
For $\eta\in T^n_{R/k}$   we let $R_\eta$ be the $k[\varepsilon]$-linear
$A_\infty$-deformation of
$R[\varepsilon]$  whose only higher multiplication is given by $\varepsilon\eta$.

\medskip

As above, for an $A_\infty$-algebra $A$ let $\Cscr^r_\infty(A)$ be the DG-category of strictly unital right $A_\infty$-modules over $A$ \cite{Lefevre}. We put $D(A)=H^\ast(\Cscr^r_\infty(A))$. This is one of the many realizations for
the derived category of an $A_\infty$-algebra (see \cite[Th\'eor\`eme 4.1.3.1(D2)]{Lefevre}) for which we consider $\Cscr^r_\infty(A)$ to be its standard enhancement.
\begin{remark}
\label{rem:trouble}
$A$ is an $A$-$A$-bimodule and hence the left $A$-action on  $A$ defines $A_\infty$-quasi-isomorphism (see \cite[Lemma 5.3.0.1]{Lefevre})
$A\r \Cscr^r_\infty(A)(A,A)$ which is however not an isomorphism.
\end{remark}
\begin{proposition} \label{lem:enhancement}
Assume $\Tscr$ is a triangulated category with arbitrary coproducts and $T$ is compact generator of $\Tscr$ such that $\Tscr(T,T)=R[\varepsilon]$. 
Assume $\aa$  is some $A_\infty$-enhancement of $\Tscr$.
Then there is an $A_\infty$-quasi-equivalence
$\aa\cong \Cscr^r_\infty(R_{\eta})$ for a suitable $\eta\in T^n_{R/k}$ which sends $T$ to an object isomorphic to $R_\eta$ 
in $H^*(\Cscr^r_\infty(R_\eta))=D(R_\eta)$
such 
that the induced map $R[\varepsilon]=\Tscr(T,T)\cong  D(R_\eta)(R_\eta,R_\eta)=R[\varepsilon]$ is the identity. 
Moreover $\eta$ is uniquely determined by the triangulated structure on $\Tscr$ and in particular is independent of the chosen quasi-equivalence.
\end{proposition}
\begin{proof} Let $\mathsf{R}=\aa(T,T)$.
By the work of Porta \cite{porta} the
$A_\infty$-functor 
\[
Y:\aa\r \Cscr^r_\infty(\mathsf{R}):X\mapsto \aa(T,X)
\]
is a quasi-equivalence which sends $T$ to $\mathsf{R}$. Indeed $\Cscr^r_\infty(\mathsf{R})$ is pre-triangulated and so is $\aa$ by the definition of enhancement.
So $H^\ast(Y)$ is exact.
Since the essential image of $H^\ast(Y)$ contains a generator of $H^\ast(\Cscr^r_\infty(\mathsf{R}))$ (namely $\mathsf{R}$)  it is sufficient to show that
$L:=H^\ast(Y)$ is fully faithful. By the Brown representability theorem $L$ has a right adjoint $R$ which moreover commutes with coproducts (this follows from the fact that $L$ send
the compact generator $T$ to the compact object $\mathsf{R}$). Hence the full subcategory of $H^\ast(\aa)$ spanned by objects $X$ such that $X\r RLX$ is an isomorphism is closed
under shifts, cones, summands and arbitrary coproducts. Moreover, applying $H^\ast(\aa)(T,-)$ we see that it contains $T$. 
Hence it must be $H^\ast(\aa)$ itself. From this one deduces that $L$ is fully faithful.

Now $\mathsf{R}$ is a DG-algebra with cohomology $R[\varepsilon]$, so it is $A_\infty$-isomorphic to a minimal $A_\infty$-structure on $R[\varepsilon]$ with $m_2$ being the usual multiplication. 
For degree reasons, the
only such $A_\infty$-structures are (up to $A_\infty$-isomorphism) of the form $R_\eta$. Hence after choosing an $A_\infty$-quasi-isomorphism $R_\eta\r \mathsf{R}$
we obtain a quasi-equivalence $\Cscr^r_\infty(\mathsf{R})\r \Cscr^r_\infty(R_\eta)$ 
which sends $\mathsf{R}$ to an object quasi-isomorphic to $R_\eta$ in a way which induces the identity on cohomology.
Composing with $Y$ completes the proof of the first part of the proposition.

For $\lambda\in k^n$ let $K_\lambda^\bullet$ be the $R$-Koszul complex on $(x_1-\lambda_1,\ldots,x_n-\lambda_n)$. This is a resolution of  $R_\lambda:=R/((x_i-\lambda_i)_i)$. Put $K^\bullet_{\lambda,T}=K^\bullet_\lambda\otimes_R T$. This is a complex in $\Tscr$.
The conditions \eqref{eq:todaexistence} and \eqref{toda:coset} hold for $K^\bullet_{\lambda,T}$ and hence 
the higher Toda bracket  $\langle K^\bullet_{\lambda,T} \rangle$ is a coset of 
$\sum_i \Tscr(T,T)_{-n+2}(x_i-\lambda_i)$ in $\Tscr(T,T)_{-n+2}=R\varepsilon$. We define $\eta_{\lambda,T}\in R_\lambda$
such that $\eta_{\lambda,T}\varepsilon$ 
is the sole element of the image of $\langle K^\bullet_{\lambda,T} \rangle$ in $R_\lambda$.

By the constructed quasi-equivalence we have 
$\eta_{\lambda,T}=\eta_{\lambda,R_\eta}$. Alas we cannot immediately apply Theorem \ref{th:toda:Ainfty} to the right-hand side of this equality as the $A_\infty$-category spanned by the terms of the complex  $K^\bullet_{\lambda,R_\eta}$ (finite direct sums of $R_\eta$) is not minimal (see Remark \ref{rem:trouble}). To work around this
let $\mathsf{S}=\Cscr_\infty^r(R_\eta)(R_\eta,R_\eta)$, which we regard as a one object $A_\infty$-category $(\mathsf{S},\bullet)$.
%Now $R_\eta$ is an $R_\eta$-$R_\eta$-bimodule and hence the left $R_\eta$-action on $R_\eta$ defines an $A_\infty$-quasi-isomorphism $R_\eta\r \mathsf{S}$
%(see \cite[Lemma 5.3.0.1]{Lefevre}).
As in Remark \ref{rem:trouble} we obtain an $A_\infty$-quasi-isomorphism $R_\eta\r \mathsf{S}$.
Composing with $(\mathsf{S},\bullet)\r \Cscr^r_\infty(R_\eta):\bullet\mapsto R_\eta$ we obtain a quasi-fully faithful $A_\infty$-functor
$(R_\eta,\bullet)\r \Cscr^r_{\infty}(R_\eta):\bullet \mapsto R_\eta$
which gives rise to a quasi-fully faithful $A_\infty$-functor
\[
\Tw R_\eta\r \Tw \Cscr^r_\infty(R_\eta)\cong \Cscr^r_\infty(R_\eta)
\]
which sends $K^\bullet_{\lambda,R_\eta}\in \Tw R_\eta$ to $K^\bullet_{\lambda,R_\eta} \in \Cscr^r_\infty(R_\eta)$. It follows that we may perform the calculation of $\eta_{\lambda,R_\eta}$ in $\Tw R_\eta$. As the $A_\infty$-subcategory of $\Tw R_\eta$ spanned by direct sums of $R_\eta$ is minimal, we are now in a position to apply Theorem \ref{th:toda:Ainfty} and we obtain 
that, up to a global sign, $\eta_{\lambda,R_\eta}$ is the image of $\sum_{\sigma\in S_n}(-1)^{\sigma}\eta(x_{\sigma(1)}-\lambda_{\sigma(1)},\dots,x_{\sigma(n)}-\lambda_{\sigma(n)})$ in $R_\lambda$. Since $T^n_{R/k}=R\bigwedge_i \partial/\partial x_i$, this is the same as the image of $n!\eta(x_1,\ldots,x_n)$.
We obtain by varying $\lambda$ that $\eta$ is uniquely determined.
\end{proof}
\begin{theorem} \label{thm:mainth}
Choose $0\neq \eta\in T^n_{R/k}$ and put $\aa=\Cscr_\infty^r(K)$, $\bb=\Cscr^r_\infty(R_\eta)$. After extending $\eta$ to $T^n_{K/k}=T^n_{R/k}\otimes_R K$, we consider $K_\eta$ as an object in $\bb$. There is an $A_{n-1}$-functor
\[
F:\aa\r \bb
\]
which sends $K$ to $K_\eta$. The corresponding functor
\[
H^\ast(F):D(K)\r D(R_\eta)
\]
does not lift to an $A_n$-functor, even after changing the enhancements on $D(K)$ and $D(R_\eta)$.
\end{theorem}
Before giving the proof of this theorem we show that it implies Theorem \ref{intro:thm0} in the introduction.
\begin{proof}[Proof of Theorem \ref{intro:thm0}]
From Theorem \ref{thm:mainth} we obtain that the hypotheses of Proposition  \ref{prop:no lift if strong unique enhancement}
are satisfied for $(\aa,\bb,F)$ (with $n$ replaced by $n-1$) and thus for $n\ge 14$ we obtain a triangulated category $\Dscr=H^\ast(\aa\coprod_{{}_F\bb}\bb)$ without $A_\infty$-enhancement
with a semi-orthogonal decomposition
\[
\Dscr=\langle D(K), D(R_\eta)\rangle.\qedhere
\]
\end{proof}
\begin{proof}[Proof of Theorem \ref{thm:mainth}]
We first discuss the construction of the functor $F$. To be compatible with Propositions \ref{prop:loc} and \ref{lem:enhancement}, put $\Tscr=D(R_\eta)=H^\ast(\Cscr^r_\infty(R_\eta))$ and let
$T$ be the object $R_\eta$. Put $S=R-\{0\}$. It is easy to that $T_S=K_\eta$. Indeed $K_\eta$ is in $\Tscr_S$ and $\cone(R_\eta\r K_\eta)$ is in $\Cscr$ by \eqref{eq:ore2}. 
In particular it follows by \eqref{eq:factor} that $\Tscr(K_\eta,K_\eta)=K[\varepsilon]$. 

Choosing homotopies we obtain an $A_2$-functor
\begin{equation}
\label{eq:A2}
F: K\r \Cscr^r_{\infty}(R_\eta):K\mapsto K_\eta
\end{equation}
and the obstructions against extending $\mu$ to an $A_i$-functor are in $\HH^j(K,\Tscr(K_\eta,K_\eta)_{-j+2})$ for $3\le j\le i$ (see e.g.\ 
\cite[Lemma 7.2.1]{RizzardoVdB2}). Since $\Tscr(K_\eta,K_\eta)=K[\varepsilon]$ the obstructions vanish for $j<n$. So $F$ extends to an $A_{n-1}$-functor. Let $\Free\tilde{}\,(-)$ be defined
as $\Free(-)$ but allowing arbitrary formal direct sums. If $\aa$ is an $A_n$-category then so is $\Free\tilde{}\,(\aa)$ and similar statement is true for functors.
We then obtain an $A_{n-1}$-functor
$\Free\tilde{}\, (F):\Free\tilde{}\, (K)\r \Free\tilde{}\, (\Cscr_\infty^r(R_\eta))$. Since $\Free\tilde{}\, (K)$ is quasi-equivalent to $\Cscr_\infty(K)$ (both are models for $D(K)$ which is semi-simple) and the direct sum
defines an $A_\infty$-functor $\Free\tilde{}\, (\Cscr_\infty^r(R_\eta))\r \Cscr_\infty^r(R_\eta)$, after choosing a suitable $A_\infty$-quasi-inverse to the first functor we obtain the sought $A_{n-1}$-functor $F:\Cscr^r_\infty(K)\r \Cscr^r_\infty(R_\eta)$ which
sends $K$ to $K_\eta$.

We claim that $F$ does not lift to an $A_n$-functor, even if we change enhancements. It if did, the $A_2$-functor \eqref{eq:A2} would also lift to an $A_n$-functor, as by Proposition \ref{lem:enhancement}
the enhancement on $\Cscr^r_{\infty}(R_\eta)$ is (weakly) unique and (as we have shown in the first paragraph) the object $K_\eta$ is determined by the triangulated structure. If this were possible then it would induce the structure of an $A_n$-functor on the corestriction
\[
K\xrightarrow{\mu}\cc\subset \Cscr^r_{\infty}(R_\eta)
\]
where $\cc$ is the full subcategory of $\Cscr^r_{\infty}(R_\eta)$ spanned by the single object $K_\eta$. Put $\mathsf{K}=\Cscr_\infty^r(R_\eta)(K_\eta,K_\eta)$. Since $K_\eta$ is an $K_\eta$-$R_\eta$-bimodule,
the left $K_\eta$-action on $K_\eta$ induces an $A_\infty$-quasi-isomorphism $K_\eta\r \mathsf{K}=\cc$. Taking an $A_\infty$-quasi-inverse and composing with $K\xrightarrow{\mu} \cc$ we obtain
and $A_n$-morphism $K\r K_\eta$ such that $H^\ast(K)\r H^\ast(K_\eta)=K[\epsilon]$ is the natural inclusion.
Such an $A_n$-morphism does not exist as $\eta\neq 0$ \cite[Chapitre B]{Lefevre}.
\end{proof}

%\bibliography{nomodel}
%\bibliographystyle{amsplain}

\providecommand{\bysame}{\leavevmode\hbox to3em{\hrulefill}\thinspace}
\providecommand{\MR}{\relax\ifhmode\unskip\space\fi MR }
\providecommand{\MRhref}[2]{%
  \href{http://www.ams.org/mathscinet-getitem?mr=#1}{#2}
}
\providecommand{\href}[2]{#2}

\end{document}